\documentclass[reqno, 11pt]{amsart}

\usepackage[letterpaper,hmargin=1in,vmargin=1.2in]{geometry}
\usepackage{amsmath, amssymb, amsthm, verbatim, url}
\usepackage{graphicx}
\usepackage{enumerate, pinlabel}
\usepackage{amsmath,amssymb,amsthm,mathrsfs,graphicx,url}
\usepackage[usenames,dvipsnames]{color}
\usepackage[colorlinks=true,linkcolor=Black,citecolor=Black, urlcolor=Black]{hyperref}

\DeclareMathOperator \re {Re}
\DeclareMathOperator \im {Im}

\DeclareMathOperator \supp {supp}

\newtheorem{prop}{Proposition}
\newtheorem{lem}[prop]{Lemma}

\newtheorem{thm}{Theorem}
\numberwithin{equation}{section}

\newcommand\imag{\mathop{\rm Im}}
\newcommand*{\defeq}{\mathrel{\vcenter{\baselineskip0.5ex \lineskiplimit0pt
                     \hbox{\scriptsize.}\hbox{\scriptsize.}}}%
                     =}

\def \Real {\mathbb R}
\def \Complex {\mathbb C}
\def \Natural {\mathbb N}
\def \mco {\mathcal O}
\def \XR {\Omega_R}
\author{T. J. Christiansen}
\address{Department of Mathematics, University of Missouri, Columbia, MO, USA}
\email{christiansent@missouri.edu}
\author{K. Datchev}
\address{Department of Mathematics, Purdue University, West Lafayette, IN, USA}
\email{kdatchev@purdue.edu}

\title[Resolvent estimates for star-shaped waveguides]{Resolvent estimates, wave decay, and resonance-free regions for star-shaped waveguides}

\date{\today}

\begin{document}

\begin{abstract}
Using coordinates $(x,y)\in \Real\times \Real^{d-1}$, we introduce the notion that an unbounded domain in $\Real^d$ is star shaped with respect to $x=\pm \infty$.  For such domains, we prove estimates on the resolvent of the Dirichlet Laplacian near the continuous spectrum.  When the domain has  infinite cylindrical ends, this has consequences for
wave decay and resonance-free regions. Our results also cover examples beyond the star-shaped case, including scattering by a strictly convex obstacle inside a straight planar waveguide.
\end{abstract}

\maketitle

\section{Introduction}

Let $\Omega \subset \mathbb R^d$, $d \ge 2$, be an open set of infinite volume, and equip the Laplacian $\Delta$ on $\Omega$ with Dirichlet boundary conditions. 
We wish to  understand how the behavior of the resolvent of the Laplacian near the spectrum is related to the geometry of $\Omega$, and to deduce consequences for wave evolution and decay, and for the distribution of resonances when these can be defined.

When $\mathbb R^d \setminus \Omega$ is bounded, this is the celebrated obstacle scattering problem. Then a particularly favorable geometric assumption, going back to the original work of Morawetz \cite{m61}, is that the obstacle is star shaped. In this paper we adapt this assumption to the study of waveguides, which are domains bounded in some directions and unbounded in others. We focus especially on domains with  cylindrical ends (which have one infinite dimension), but our resolvent estimates hold for domains with more general  ends. Our results also cover the problem of scattering by a strictly  convex obstacle inside a straight planar waveguide (see Figure~\ref{f:co}) for which we  prove a resolvent estimate in Theorem~\ref{t:convexobs}, wave decay in Theorem~\ref{thm:intro2}, and a resonance-free region in Theorem~\ref{t:resfree}.

\begin{figure}[h]
\labellist
\pinlabel $\Omega$ [l] at 600 170 
\endlabellist
\includegraphics[width=10cm]{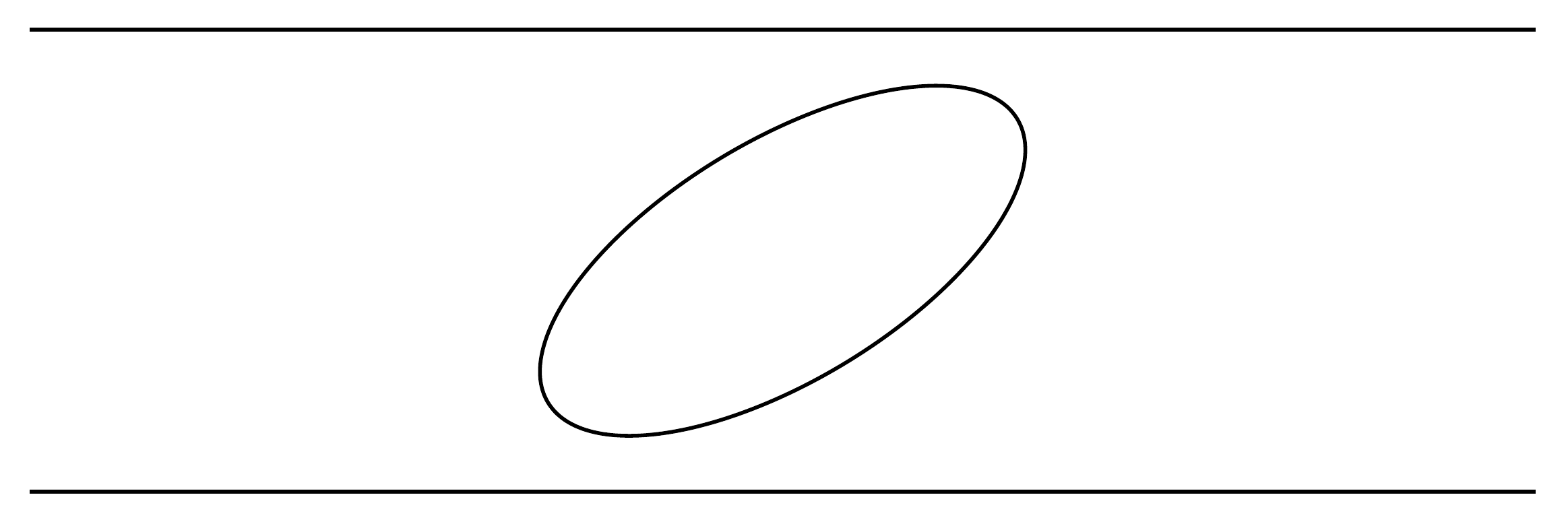}
 \caption{A strictly  convex obstacle inside a straight planar waveguide.}\label{f:co}
\end{figure}

Our analysis is based on the following definition, which applies to some but not all strictly  convex obstacles inside a straight planar waveguide.

\noindent\textbf{Definition.} We say that $\Omega$ is \textit{star shaped with respect to $x = \pm \infty$} if
\begin{equation}\label{e:xnux}
 x \nu_x \le 0 \text{ throughout } \partial \Omega,
\end{equation}
where $(x,y) = (x, y_1,  \dots, y_{d-1})$ are Cartesian coordinates on $\mathbb R^d$ and $\nu = (\nu_x, \nu_{y_1}, \dots \nu_{y_{d-1}})$ is the outward pointing unit normal  vector to $\partial \Omega$. 

In other words, \eqref{e:xnux} says that, if $p \in \partial \Omega$ does not lie in the coordinate hyperplane $x=0$, then $\nu$ at $p$  either points toward $x=0$
or is parallel to $x=0$.  See Figures~\ref{f:cig},~\ref{f:hour}, and~\ref{f:flat} below for examples of such domains with $d=2$; in Figure~\ref{f:cig} the vertical axis $x=0$ is to the left of the domain, and in Figures~\ref{f:hour} and~\ref{f:flat} the axis is in the center. Note that the example in Figure~\ref{f:co} does \textit{not} obey \eqref{e:xnux}.

This is to be compared with the analogous assumption in  scattering by a compact obstacle, as in \cite[(11)]{m61}, which says that
\begin{equation}\label{e:xynu}
x \nu_x + y_1 \nu_{y_1} + \cdots + y_{d-1}\nu_{y_{d-1}} \le 0  \text{ throughout } \partial \Omega.
\end{equation}
In other words, \eqref{e:xynu} says that the obstacle $\mathbb R^d \setminus \Omega$ is star shaped.

Note that \eqref{e:xynu} implies that there are no trapped billiard trajectories in $\Omega$; that is, all billiard trajectories go to infinity forward and backward in time. By contrast, our assumption \eqref{e:xnux} allows trapping, and indeed all domains with  cylindrical ends have trapped trajectories. In some sense, among domains with  cylindrical ends, those obeying \eqref{e:xnux} have the least trapping possible. In Figures \ref{f:cig}, \ref{f:hour}, and \ref{f:flat} below, the trapped trajectories are the vertical bouncing ball orbits between points where the boundary is horizontal.

Fundamental to the proofs of the resolvent estimates we give under the assumption \eqref{e:xnux}  are some integration by parts identities.  By combining the assumption (1.1) on the boundary $\partial \Omega$ with the Dirichlet boundary condition, we obtain identities in which the boundary term has favorable sign. This part of our proofs is analogous to \cite[(16)]{m61}.

We first present our resolvent estimates, which hold for domains with rather general infinite ends. Afterwards we give consequences for wave evolution and analytic continuation of the resolvent, under the additional assumption that $\Omega$ has  cylindrical ends.

\subsection{Resolvent estimates} 
The spectrum of $-\Delta$, with Dirichlet boundary conditions, is contained in $[0,\infty)$, and for $z \in \mathbb C$ not in the spectrum, let
\[
(-\Delta - z)^{-1} \colon L^2(\Omega) \to L^2(\Omega),
\]
be the corresponding resolvent. Throughout the paper we assume that $\partial \Omega$ is  Lipschitz, and that every point $p$ on $\partial \Omega$ has a neighborhood $U_p$ such that either $U_p \cap \Omega$ is convex or $U_p \cap \partial \Omega$ is $C^{1,1}$.

Our strongest result holds in the case that the domain has only one infinite end:

\begin{thm}\label{t:cig}
Suppose that $\Omega$ satisfies the assumption \eqref{e:xnux}   and that $x>0$ throughout $\Omega$. Then for any $\delta \in (0,1]$ and $z \in \mathbb C \setminus [0,\infty)$ we have
 \begin{equation}\label{e:recig}
    \|(1+x)^{-\frac{3+\delta}2}(-\Delta - z)^{-1}(1+x)^{-\frac{3+\delta}2}\|_{L^2(\Omega) \to L^2(\Omega)}  \le \frac 3 \delta(1 + |z|^{1/2}).
 \end{equation}
\end{thm}

Examples include cigar-shaped domains such as the union of the ball $\{(x,y) \colon (x-1)^2 + |y|^2 <1\}$ with the half-cylinder $\{(x,y) \colon x > 1 \text{ and } |y|<1\}$, the parabolic domain $\{(x,y) \colon x>|y|^2\}$, and more generally any epigraph $\{(x,y) \colon x > f(y)\}$ where $f \in C^{1,1}(\mathbb R^{d-1})$ is nonnegative. See also Figure~\ref{f:cig}.

\begin{figure}[h]
\labellist
\pinlabel $\Omega$ [l] at 300 190
\pinlabel $\Omega$ [l] at 820 190
\pinlabel $\Omega$ [l] at 1330 170
\endlabellist
\includegraphics[width=4cm]{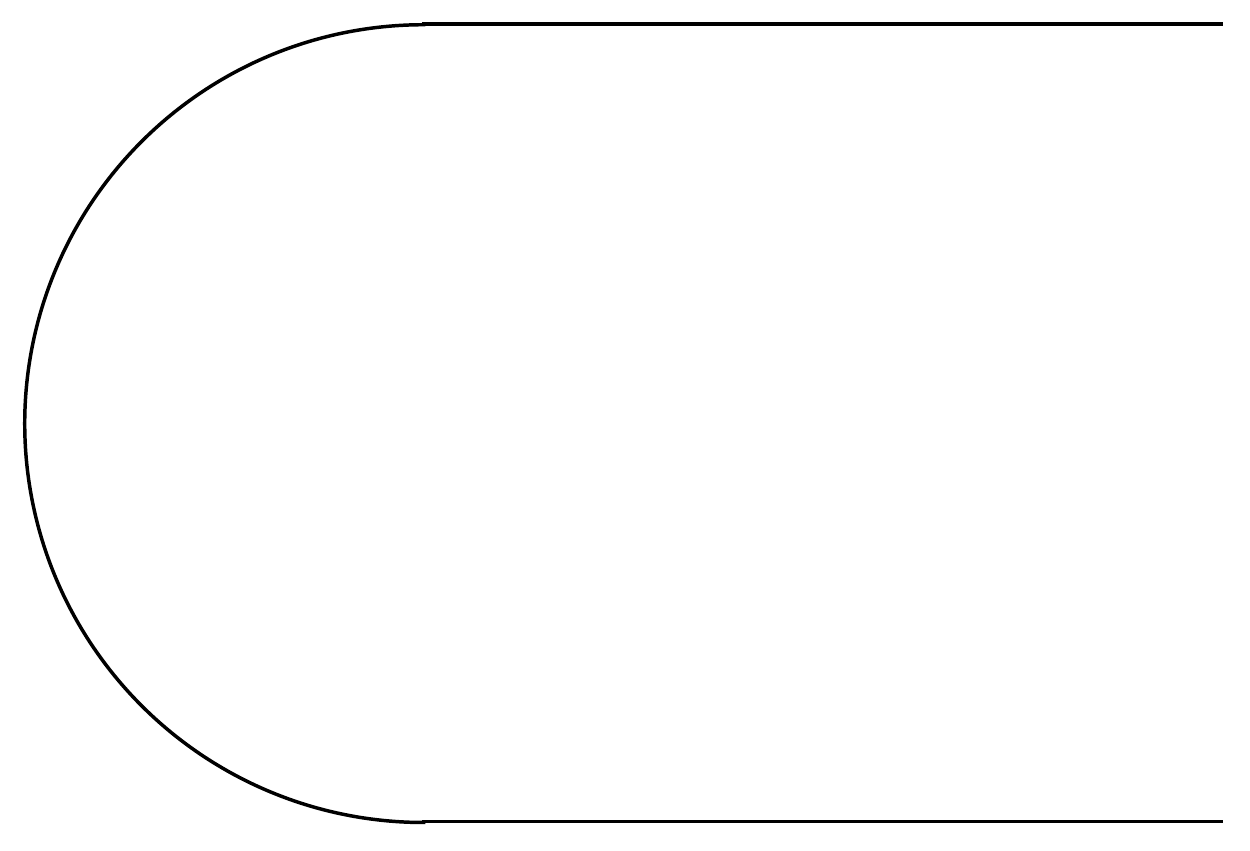}
\hspace{1.5cm}
\includegraphics[width=4cm]{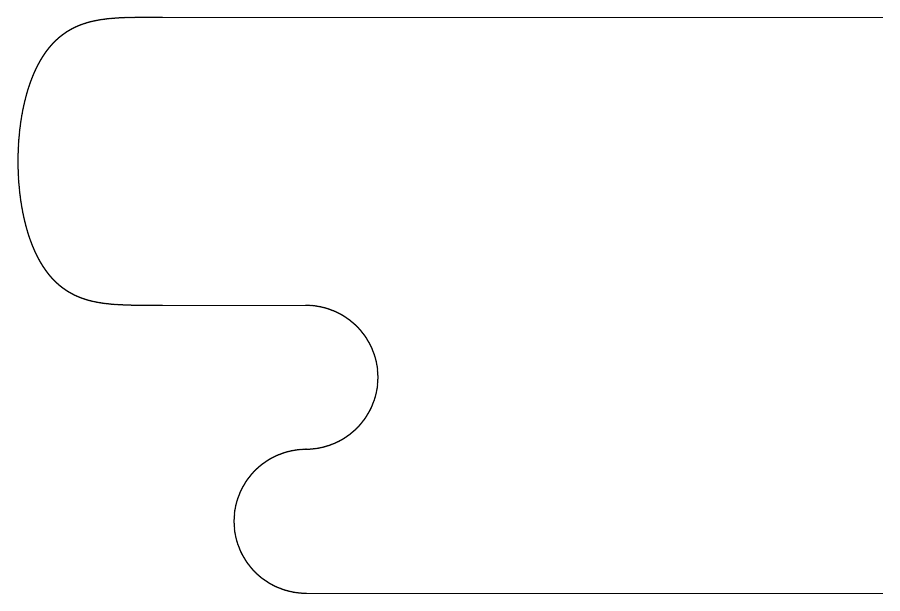} \hspace{1.5cm}
\includegraphics[width=4cm]{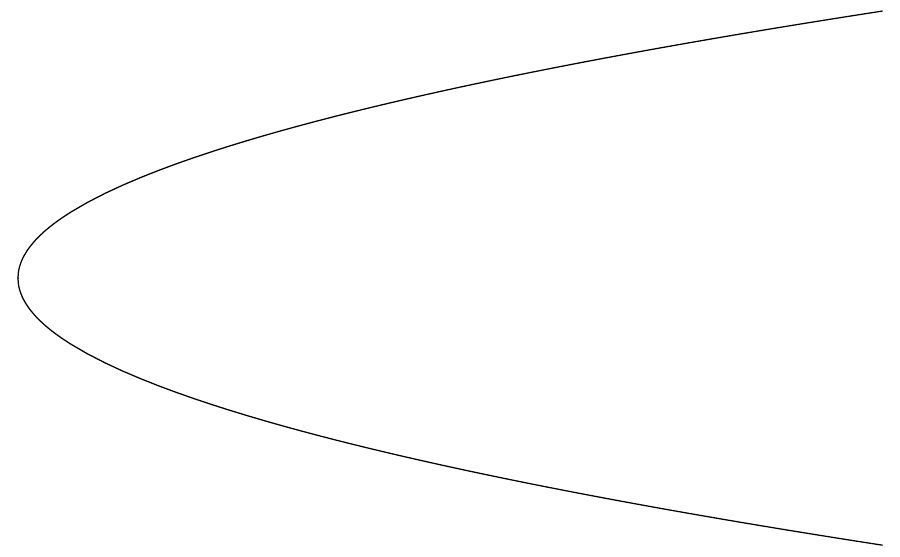} 
 \caption{Some domains to which Theorem \ref{t:cig} applies. The first two have  cylindrical ends, and the third does not.}\label{f:cig}
\end{figure}

A weaker version of Theorem \ref{t:cig} holds in the presence of multiple infinite ends, under a `flaring' condition; namely when there is a suitable region where $x\nu_x$ is bounded away from zero. Note that some such condition is needed, because if $x \nu_x \equiv 0$ then by separation of variables and direct computation (as in Section 1.1 of \cite{cd2}) one checks that there are infinitely many resonances embedded in the continuous spectrum.

\begin{thm}\label{t:hour}
Suppose that $\Omega$ satisfies the assumption \eqref{e:xnux}  and there is an open interval $I$  and a positive constant $C_I$ such that
\begin{equation}\label{e:flare}
x \nu_x \le -C_I,
\end{equation}
on the intersection of $\partial \Omega$ with $I \times \mathbb R^{d-1}$.  Suppose further that the intersection of $\Omega$ with $I \times \mathbb R^{d-1}$ is bounded. Then for any $\delta >0$ there are positive constants $E_0$ and $C$ such that
 \begin{equation}\label{e:e12}
    \|(1+|x|)^{-\frac{3+\delta}2}(-\Delta - E - i \varepsilon)^{-1}(1+|x|)^{-\frac{3+\delta}2}\|_{L^2(\Omega) \to L^2(\Omega)}  \le  C E^{1/2}
 \end{equation}
for all $E \ge E_0$ and $\varepsilon \in (0,1]$.
\end{thm}
Examples include hourglass-shaped domains like $\{(x,y) \colon |y| < f(x)\}$ for some nonconstant $f \in C^{1,1}(\mathbb R)$ satisfying $x f'(x) \ge 0$ for all $x$. See also Figure~\ref{f:hour}.

\newcommand*{\vcenteredhbox}[1]{\begin{tabular}{@{}c@{}}#1\end{tabular}}

\begin{figure}[h]
\labellist
\pinlabel $\Omega$ [l] at 300 150
\pinlabel $\Omega$ [l] at 815 122
\pinlabel $\Omega$ [l] at 1340 155
\endlabellist
\vcenteredhbox{\includegraphics[width=4cm]{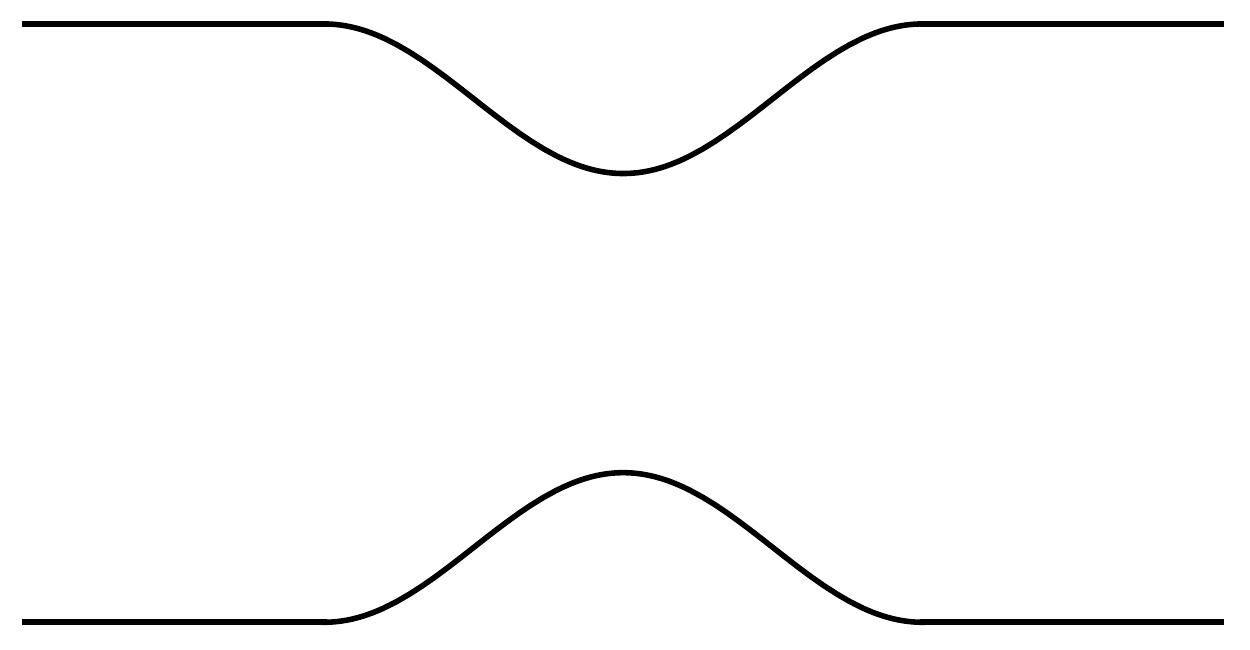}}
\hspace{1.5cm}
\vcenteredhbox{\includegraphics[width=4cm]{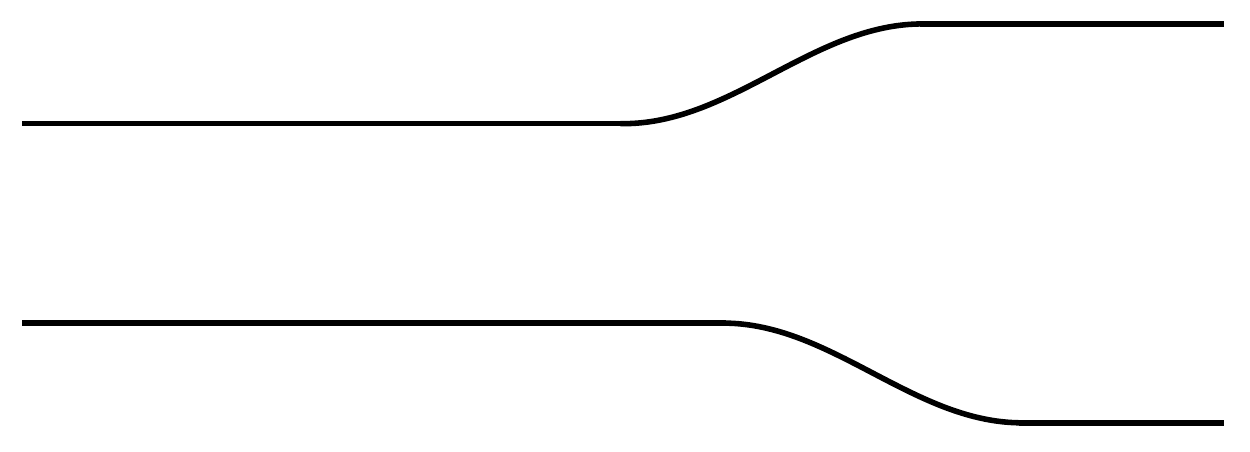}} 
\hspace{1.5cm}
\vcenteredhbox{\includegraphics[width=4cm]{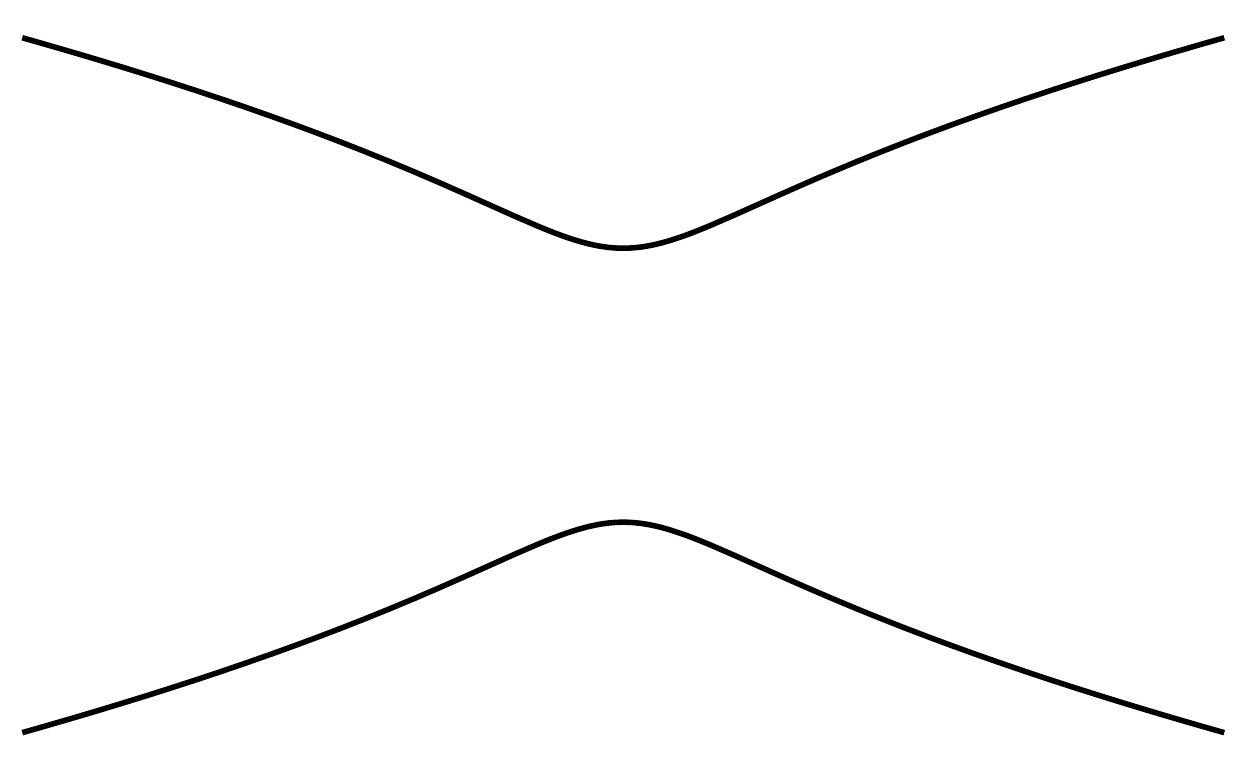}} 
 \caption{Some domains to which Theorem \ref{t:hour} applies. The first two have  cylindrical ends and the third does not.}\label{f:hour}
\end{figure}

For planar domains, we only need the flaring requirement \eqref{e:flare} on part of the intersection of $\partial \Omega$ with $I \times \mathbb R$:

\begin{thm}\label{t:flat}
Suppose that $\Omega$ satisfies the assumption \eqref{e:xnux} and that $d=2$. Let $I$ be an open interval and let $C_I$ be a positive constant. Let $\Gamma_F$ be part of the intersection of $\partial \Omega$ with $I \times \mathbb R$ on which the flaring requirement \eqref{e:flare} holds. Suppose that the intersection of $\Omega$ with $I \times \mathbb R$ consists of bounded open sets $\Omega_1, \dots, \Omega_K$ with mutually disjoint closures such that for each $k = 1, \dots K$, 
\[(\partial \Omega \cap \partial \Omega_k) \setminus \Gamma_F \subset I \times \{a_k\},\] for some real $a_k$.
Then for any $\delta >0$ there are positive constants $E_0$ and $C$ such that \eqref{e:e12} holds for all $E \ge E_0$ and $\varepsilon \in (0,1]$.
\end{thm}

Some examples of domains for which Theorem \ref{t:flat} applies are
shown in Figure \ref{f:flat}.  One class of such examples
is that of 
 straight planar waveguides with suitable convex obstacles 
$\left(\mathbb R \times (-1,1)\right)  \setminus \overline{\mco}$, where 
$\overline{\mco} \subset \mathbb R \times (-1,1)$ is a convex closed set such that the maximum and minimum values of $y$ on $\overline{\mco}$ are both attained on the axis $x=0$. In Theorem~\ref{t:convexobs} we prove the corresponding result for more general convex $\overline{\mco}$, as in Figure~\ref{f:co}.  These 
domains do not necessarily satisfy (\ref{e:xnux}), but the proofs of Theorems
\ref{t:flat} and \ref{t:convexobs} are similar.

\begin{figure}[h]
\labellist
\pinlabel $\Omega$ [l] at 300 170
\pinlabel $\Omega$ [l] at 820 180
\pinlabel $\Omega$ [l] at 1330 187
\endlabellist
\includegraphics[width=4cm]{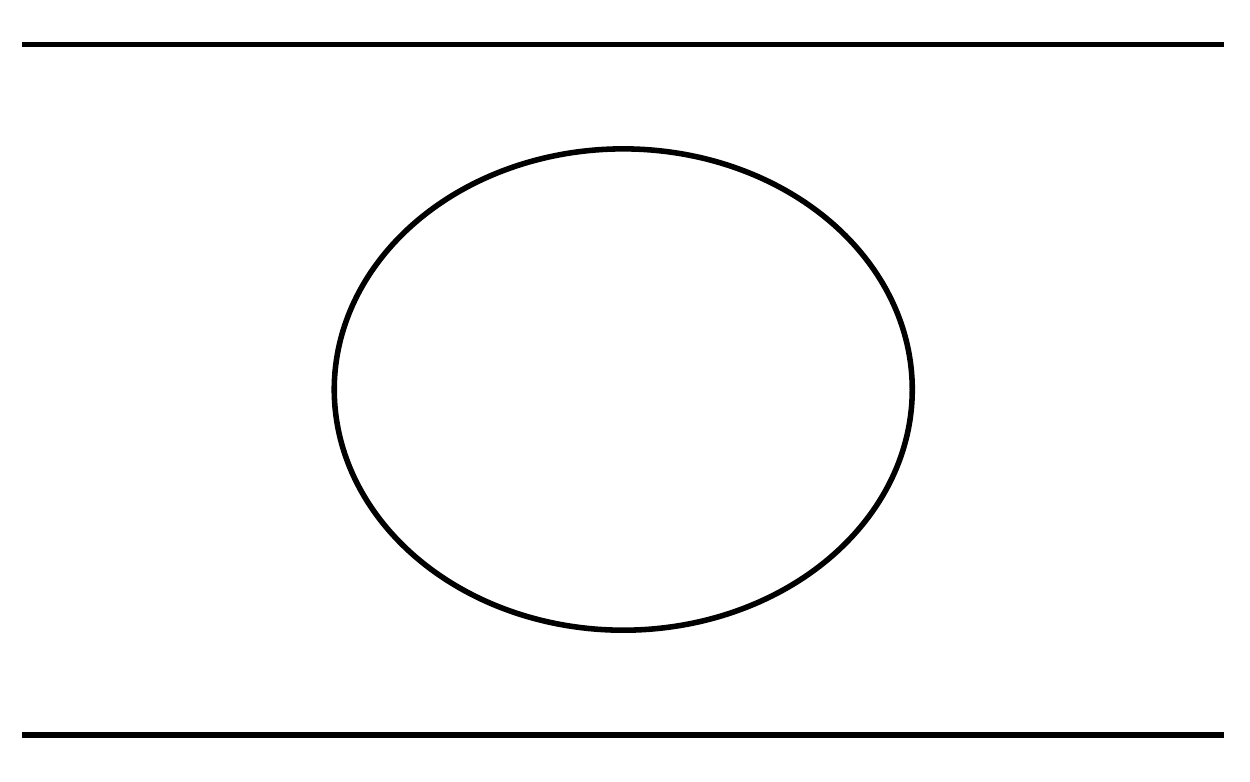}
\hspace{1.5cm}
\includegraphics[width=4cm]{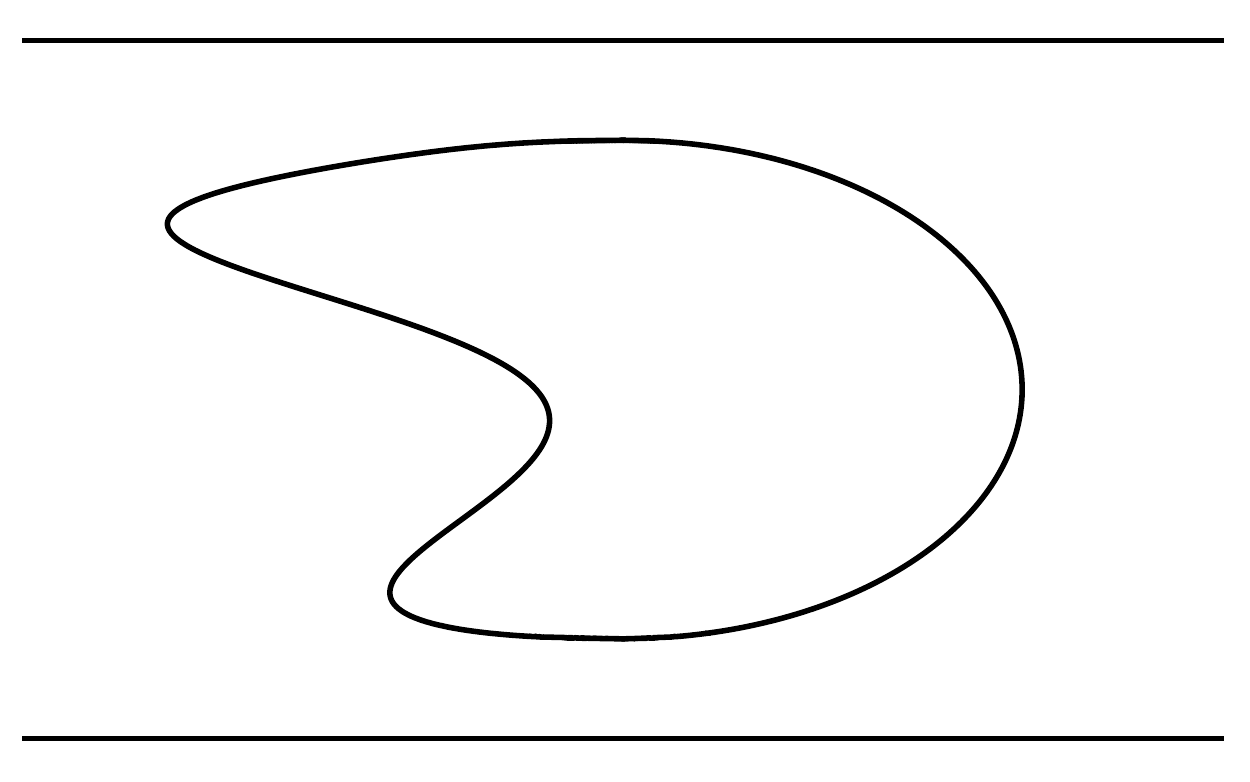}
\hspace{1.5cm}
\includegraphics[width=4cm]{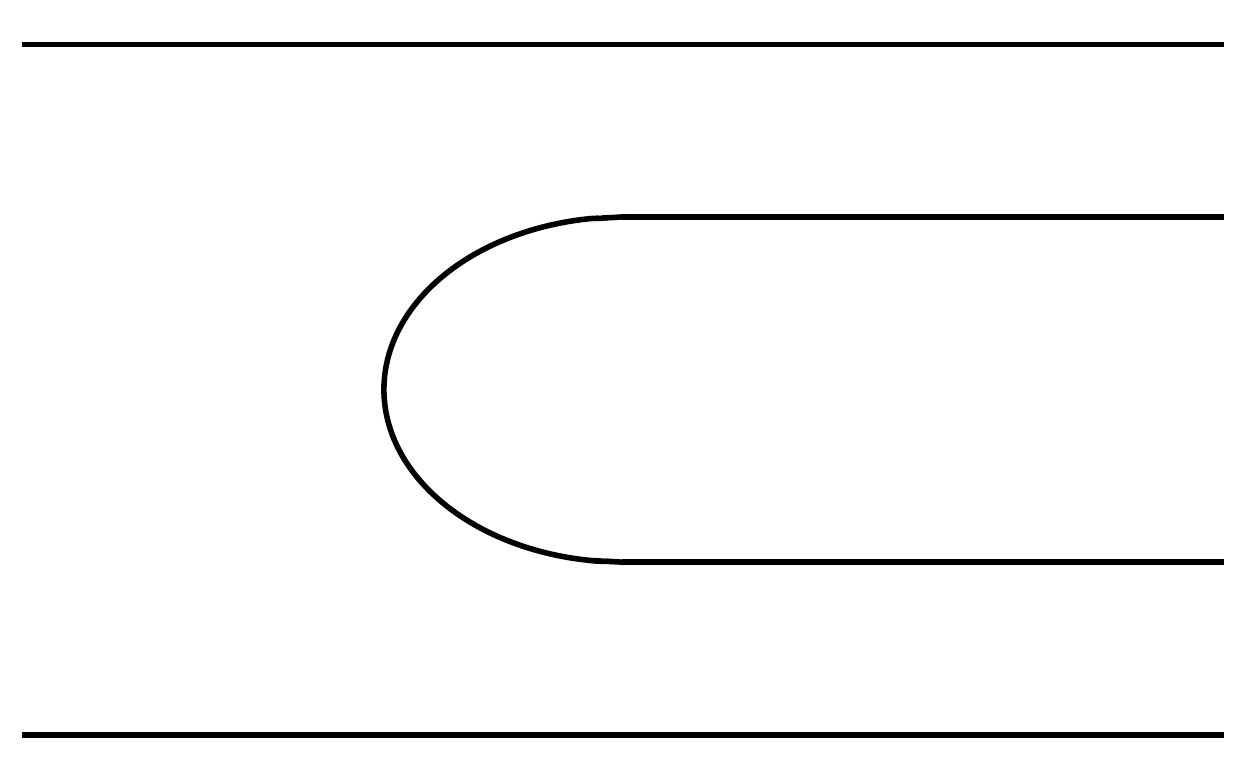}  
 \caption{Some domains with cylindrical ends to which Theorem \ref{t:flat} applies. }\label{f:flat}
\end{figure}


\begin{thm}\label{t:convexobs}
Let $\Omega=(\Real \times (-1,1))\setminus \overline{\mco}$, where 
$\mco$ 
is a non-empty open bounded strictly convex set with $C^{1,1}$ boundary
and $\overline{\mco}\subset
\Real \times (-1,1)$.  
Then for any $\delta>0$ there are positive constants $E_0$ and $C$
such that (\ref{e:e12}) holds for all $E \ge E_0$ and $\varepsilon \in (0,1]$.
\end{thm}

\subsection{Wave asymptotics and absence of eigenvalues and embedded resonances}
In this section we assume in addition that $\Omega$ has \textit{cylindrical ends}.
Specializing to our setting, by this we mean that there is $R_0>0$ such that $\Omega \cap ([-R_0,R_0] \times \mathbb R^{d-1})$ is bounded and 
\begin{equation}\label{eq:cylend}
\Omega \cap ((-\infty,-R_0] \times \mathbb R^{d-1}) = (-\infty,-R_0] \times Y_-, \qquad \Omega \cap ([R_0,\infty) \times \mathbb R^{d-1}) = [R_0, \infty) \times Y_+,
\end{equation}
where $Y_-$ and $Y_+$ are (not necessarily connected) bounded open sets in $\mathbb R^{d-1}$.  We allow the possibility that one, but not both, of 
$Y_{\pm}$ is the empty set.

Let $f_1, \ f_2 \in C_c^\infty(\Omega)$, and let $u$  solve
\begin{equation}\label{e:weq}
 (\partial_t^2 - \Delta)u = 0, \qquad (u,\partial_tu)|_{t=0} = (f_1,f_2), \qquad u|_{\partial \Omega} = 0.
\end{equation}

We prove decay rates and asymptotics for such $u$, using results of \cite{cd2}. We begin with a wave decay rate.

\begin{thm}\label{thm:t1intro} 
Suppose that $\Omega$ satisfies the assumptions of Theorem \ref{t:cig}, Theorem \ref{t:hour}, Theorem \ref{t:flat}, or of Theorem \ref{t:convexobs}. Suppose additionally that $\Omega$ has  cylindrical ends.  Let
 $ \ f_1, \ f_2 \in C_c^\infty( \Omega)$ be given,  and let $u(t)$ solve \eqref{e:weq}. Then for any $\chi \in C_c^\infty( \Omega)$ and for any $m \in \mathbb N$ there is a constant $C$ such that
 \[
 \|\chi (u(t) - u_p(t))\|_{H^m(\Omega)} \le C t^{-1} \text{ for }t \text{ sufficiently large,}
 \]
 where $u_p(t)$ is a term corresponding to the projection of the initial data $(f_1,f_2)$ onto any eigenvalues and embedded resonances of the Dirichlet Laplacian on $\Omega$. If \eqref{e:xnux} holds, then there are no such eigenvalues and embedded resonances, and $u_p(t) \equiv 0$.
\end{thm}

For a more detailed description of $u_p(t)$, see Theorem 1.1 of \cite{cd2}. The only case of Theorem \ref{thm:t1intro}  in which we do not show $u_p(t) \equiv 0$ is that of a convex obstacle inside a straight planar waveguide such that \eqref{e:xnux} does not hold.

To state our next result, let $Y$ be the disjoint union of $Y_-$ and $Y_+$, let $\Delta_Y$ be the Dirichlet Laplacian on $Y$, and let $\{\phi_j\}_{j=0}^\infty$ be a complete orthonormal set of eigenfunctions of $\Delta_Y$, with corresponding eigenvalues $\sigma_j^2$, so that
\begin{equation}\label{e:phisig}
-\Delta_Y \phi_j = \sigma_j^2 \phi_j, \qquad 0 < \sigma_1 \le \sigma_2 \le \cdots.
\end{equation} 

We get an improvement of Theorem \ref{thm:t1intro} under an additional assumption on the eigenvalues of $-\Delta_Y$.  The assumption is that there are positive constants $c_Y$ and $N_Y$, such that
\begin{equation}\label{eq:distinct}
 \sigma_{j'} - \sigma_{j} \ge c_Y \sigma_j^{-N_Y},
\end{equation}
 whenever $\sigma_{j'} > \sigma_j$. Note that this assumption allows the eigenvalues of $-\Delta_Y$ to have high multiplicities, but forbids \textit{distinct} eigenvalues from clustering too closely together.  

\begin{thm}\label{thm:intro2}
Suppose that the assumptions of Theorem \ref{thm:t1intro} hold, and also that \eqref{eq:distinct} holds.
Then for each $k_0 \ge2$ we can write
\[
 u(t)  = u_p(t) +  \sum_{k=1}^{k_0-1}t^{-1/2-k} 
\sum_{j=1}^\infty (e^{it \sigma_j}b_{j,k,+}+ e^{-it \sigma_j}b_{j,k,-}) + u_{r,k_0}(t),
\]
for some  $b_{j,k,\pm}\in C^\infty(\Omega)$, where $u_p(t)$ is as in Theorem \ref{thm:t1intro}, and where for any 
$\chi \in C_c^\infty(\Omega)$  and $m \in \mathbb N$ there is a constant $C$ so that
$$\sum_{j=0}^\infty \|\chi  b_{j,k,\pm}\|_{H^m(\Omega)}<+\infty,\; k=1,2,...,k_0-1,$$
and
\[
 \|\chi u_{r,k_0}(t)\|_{H^m(\Omega)} \le C t^{-k_0}\; \text{for $t$ sufficiently large}.
\] 
If \eqref{e:xnux} holds, then $u_p(t) \equiv 0$.
\end{thm}
For a more detailed description of the $b_{j,k,\pm}$, see Theorem 1.2 and Lemma 4.7 of \cite{cd2}.
In particular, Theorem \ref{thm:intro2} shows that 
\[
\|\chi (u(t)-u_p(t))\| \le C t^{-3/2}.
\]
This is sharp when $\Omega =(0,\infty) \times Y$, where $Y \subset \mathbb R^{d-1}$ is bounded, by the computation in Section 1.1 of \cite{cd2}, in particular equation (1.6) there.

The fact that if \eqref{e:xnux} holds, then $u_p(t) \equiv 0$ in Theorems~\ref{thm:t1intro} and~\ref{thm:intro2}, depends on the following result ruling out eigenvalues and real resonances.  Although this is perhaps well known
(see, for example, \cite[Theorem 3.1]{MoWe} for a similar 
result), we include a proof both for completeness and because it uses
an integration by parts identity similar  to that 
used in the proofs of  Theorems \ref{t:cig}, \ref{t:hour}, \ref{t:flat},
and \ref{t:convexobs}.

\begin{thm}\label{t:noembed}
Suppose that $\Omega$ satisfies the assumption \eqref{e:xnux} and has  cylindrical ends. Then the Dirichlet Laplacian on $\Omega$
has no eigenvalues.  For such $\Omega$, the Dirichlet Laplacian on $\Omega$ has resonances embedded in the continuous spectrum if and only if  $\Omega$ is  the product $\Real \times \widetilde{Y}$
for some set $\widetilde{Y}\subset \Real^{d-1}$.


\end{thm}

By separation of variables and direct computation (as in Section 1.1 of \cite{cd2}) one checks that if  $\Omega = \Real \times \widetilde{Y}$ then the Dirichlet Laplacian on $\Omega$
has threshold resonances at every point in the Dirichlet spectrum of
$-\Delta_{\widetilde{Y}}$.  Theorem \ref{t:noembed} shows that no other 
sufficiently regular
domains $\Omega$ with  cylindrical ends obeying \eqref{e:xnux}  can have any poles of the Dirichlet resolvent on the real axis. Theorem \ref{t:noembed} is  a consequence of Theorem \ref{t:cig}
in case $\Omega$ has only one end; we prove the general case in Section \ref{s:noembed}. 

The proofs of Theorems  \ref{thm:t1intro} and \ref{thm:intro2} depend upon results on resonance-free regions in a neighborhood of the spectrum. Because these are more complicated to state, we present them below in Theorem~\ref{t:resfree} in Section~\ref{s:resfree}. Once Theorems  \ref{t:noembed} and \ref{t:resfree} are established, Theorems \ref{thm:t1intro} and \ref{thm:intro2} are direct consequences of Theorems 3.2 and 4.1 of \cite{cd2} (see also Theorems 1.1 and 1.2 of \cite{cd2}).


\subsection{Background and context}

A wave decay rate for star-shaped compact obstacles was proved by Morawetz in \cite{m61}, and the results there were refined and extended in many papers, including \cite{lmp63, m72, r78}, and more recently revisited and adapted to hyperbolic scattering by Hintz and Zworski in \cite{hz1, hz2}.

Our results here build on those in \cite{cd, cd2} for manifolds with cylindrical ends, which in turn are based on the spectral and scattering theory of waveguides and manifolds with  cylindrical ends developed in \cite{gold, lyf, gui, tapsit, chr95,par}. The main novelty in the present paper is the resolvent estimates in 
Theorems \ref{t:cig}, \ref{t:hour}, \ref{t:flat}, and \ref{t:convexobs}. 
These rely on integration by parts identities in the spirit of Morawetz \cite{m61}.

Waveguides appear in models of
electron motion in semiconductors and of propagation of electromagnetic and sound waves; see for example   \cite{lcm,rai,rbbh,EK,bgw}. There are many results establishing the existence of eigenvalues for waveguides, under suitable geometric conditions. Something of a survey can be found in \cite{kk}. The result in \cite{bgrs} holds in a setting in some sense opposite to ours, and shows in particular that if $\Omega \subset \mathbb R^2$ has  cylindrical ends and obeys $x\nu_x \ge0$, $x \nu_x \not\equiv 0$, then there is at least one eigenvalue. There are nonexistence results for eigenvalues in \cite{MoWe,  dp, bdk}, and another for a resonance at the bottom of the spectrum in \cite{gj}. Some weaker wave decay results (expansions up to $o(1)$ as $t \to \infty$) for planar waveguides can be found in \cite{lyf, hw}.

The resonance-free region we establish in Theorem \ref{t:resfree} is a close analogue of a corresponding region for manifolds with cylindrical ends established in \cite{cd}, and relies on  a resolvent identity due to Vodev \cite{v}. An existence result for resolvent poles (in the presence of appropriate quasimodes) on waveguides can be found in \cite{e}. Upper bounds on the number resonances for manifolds with  cylindrical ends  are given in \cite{c0}. 

Previous work also shows that our results cannot carry over directly to the case of Neumann boundary conditions. For example, in \cite{dp}, examples are given of domains $\Omega$ with cylindrical ends satisfying the hypotheses of Theorems~\ref{t:flat} and \ref{t:noembed} but whose Neumann Laplacians have eigenvalues. More simply,  if $\Omega = (0,\infty) \times (-1,1)$, then $\Omega$ satisfies the hypotheses of Theorem~\ref{t:cig}, but the Neumann Laplacian has infinitely many embedded resonances (see Section 1.1 of \cite{cd2}).

\subsection{Outline} In Section \ref{s:not} we review background regarding Sobolev spaces and establish notation. In Section \ref{s:cigar} we prove Theorem \ref{t:cig}. In Section \ref{s:hour} we prove Theorems \ref{t:hour}, \ref{t:flat}, and \ref{t:convexobs}. In Section \ref{s:noembed} we prove Theorem \ref{t:noembed}. In Section \ref{s:resfree} we obtain a resonance-free region in a neighborhood of the spectrum.

\section{Preliminaries and notation}\label{s:not}

Throughout the paper we assume that $\partial \Omega$ is  Lipschitz, and that every point $p$ on $\partial \Omega$ has a neighborhood $U_p$ such that either $U_p \cap \Omega$ is convex or $U_p \cap \partial \Omega$ is $C^{1,1}$. 

We denote by $C_c^\infty(\Omega)$ the space of functions in $C^\infty(\mathbb R^d)$ with compact support in $\Omega$, and by $C_c^\infty(\overline{\Omega})$ the space of  restrictions to $\overline \Omega$ of functions in $C^\infty(\mathbb R^d)$ with compact support in $\mathbb R^d$. We  use three different kinds of Sobolev spaces on $\Omega$. We denote by $H^k(\Omega)$ the Sobolev space of functions in $L^2(\Omega)$  whose partial derivatives up to $k$th order are in $L^2(\Omega)$, and by $H^k_0(\Omega)$ the closure of $C_c^\infty(\Omega)$ in $H^k(\Omega)$. 
We denote by $H^k_{\text{comp}}(\Omega)$ the space of functions in $H^k(\Omega)$ with compact support in $\overline \Omega$ (by \cite[Theorem 1.4.3.1]{gris} this is the same as the space of restrictions to $\Omega$ of compactly supported functions in $H^k(\mathbb R^d)$), and similarly by $L^p_{\text{comp}}(\Omega)$ the space of functions in $L^p(\Omega)$ with compact support in $\overline \Omega$.

We  integrate by parts using Green's theorem (see \cite[Theorem 1.5.3.1]{gris}). We use the fact that $C_c^\infty(\overline{\Omega})$ is dense in $H^k_{\text{comp}}(\Omega)$ (see \cite[Theorem 1.4.2.1]{gris}), and that the trace map $H^1_{\text{comp}}(\Omega) \to L^2(\partial \Omega)$ is continuous (see \cite[Theorem 1.5.1.3]{gris}).

We define the Dirichlet resolvent by taking the Friedrichs extension  of $\Delta$ with domain $C_c^\infty(\Omega)$ (see pages 82 and 83 of \cite[Chapter 8, Section 2]{tay2}). For $z \not\in [0,\infty)$ we have
\begin{equation}\label{e:resmap}
(-\Delta - z)^{-1} \colon L^2(\Omega) \to \mathcal D := \{u \in H^1_0(\Omega) \colon \Delta u \in L^2(\Omega)\}.
\end{equation}
We denote by $\mathcal D_{\text{comp}}$  the set of functions in $\mathcal D$ with compact support in $\overline \Omega$. The regularity assumption on $\partial \Omega$ is made so as to ensure that 
\begin{equation}\label{e:dh2}
 \mathcal D_{\text{comp}} = H^2_{\text{comp}}(\Omega) \cap H^1_0(\Omega).
\end{equation}
Near points on $\partial \Omega$ where $\partial \Omega$ is $C^\infty$, \eqref{e:dh2} follows from \cite[Chapter 5, Theorem 1.3]{tay}). Near points where $\partial \Omega$ is $C^{1,1}$,  \eqref{e:dh2} follows from \cite[Corollary~2.2.2.4]{gris}. Near points where  $\Omega$   is convex,  \eqref{e:dh2}  follows from \cite[Theorem~3.2.1.2]{gris} (see also \cite[Chapter~5, Section~5, Exercise~7]{tay}).

For real $E$ and $\varepsilon$ we write for brevity
\[
P = P(E, \varepsilon) = - \Delta - E - i \varepsilon.
\]

We use $\|\cdot\|$ and $\langle \cdot, \cdot \rangle$ to denote the norm and inner product on $L^2(\Omega)$, and prime to denote differentiation with respect to $x$.

\section{Domains with one end}\label{s:cigar}

We begin the proof of Theorem \ref{t:cig} with an integration by parts identity in the spirit of Morawetz and others. This identity, along with some variants of it, also plays a central role in the proofs of Theorems~\ref{t:hour}, \ref{t:flat}, \ref{t:convexobs}, and \ref{t:noembed}.

\begin{lem}\label{l:ibpcig}
Let $\Omega \subset \mathbb R^d$ be an open set such that every point $p$ on $\partial \Omega$ has a neighborhood $U_p$ such that either $U_p \cap \Omega$ is convex or $U_p \cap \partial \Omega$ is $C^{1,1}$. Let $w \in C^3(\mathbb R)$ be real valued, and suppose  $w, \ w', \ w'', \ w'''$ are all bounded. Let $u \in \mathcal D$ and let $E, \ \varepsilon \in \mathbb R$. Then
\begin{equation}\label{e:ibpm}
  \langle w'u',u'\rangle  = \frac 1 4 \langle w''' u,u\rangle + \frac 12 \re \langle Pu,(wu)'\rangle  + \frac 12 \re \langle w u',Pu\rangle +  \varepsilon \im \langle wu', u \rangle + \frac 12  \int_{\partial \Omega} w |\partial_\nu u|^2 \nu_x,
\end{equation}
where $w=w(x)$.
\end{lem}

\begin{proof}
Let $u, \ v \in C_c^\infty(\overline{\Omega})$. We use a positive commutator argument with $w\partial_x$ as commutant. Computing this commutator two ways we have
\begin{equation}\label{e:wcomm}
 \langle [w\partial_x,\partial_x^2] u,v\rangle = -2\langle w'u'',v\rangle - \langle w''u',v\rangle = 2 \langle w'u',v'\rangle + \langle w''u',v\rangle - 2 \int_{\partial \Omega} w'u'\bar v \nu_x,
\end{equation} 
and
\begin{equation}\label{e:pcom1}
  \langle [w\partial_x,\partial_x^2] u,v\rangle =  \langle [P,w\partial_x] u,v\rangle = \langle Pw u',v\rangle - \langle w Pu',v\rangle.
\end{equation}
We write the right hand side of \eqref{e:pcom1} in terms of $Pu$ and $Pv$ by integrating by parts to obtain
\begin{equation}\label{e:pcom2}
 \langle Pw u',v\rangle =  \langle w u',(P+2i\varepsilon)v\rangle + \int_{\partial \Omega} \left( wu'\partial_\nu \bar v  - \partial_\nu(wu')\bar v\right),
\end{equation}
and
\begin{equation}\label{e:pcom3}
- \langle w Pu',v\rangle = \langle Pu,(wv)'\rangle - \int_{\partial \Omega} w(Pu) \bar v\nu_x.
\end{equation}
Combining \eqref{e:pcom1}, \eqref{e:pcom2}, and \eqref{e:pcom3} gives
\[
 \langle [w\partial_x,\partial_x^2] u,v\rangle = \langle Pu,(wv)'\rangle  +  \langle w u',Pv\rangle - 2 i \varepsilon \langle wu',v \rangle + \int_{\partial \Omega} \left( wu'\partial_\nu \bar v  - \partial_\nu(wu')\bar v - w (Pu)\bar v \nu_x\right).
\]
and combining also with \eqref{e:wcomm} gives
\begin{equation}\label{e:ibpuvlong}\begin{split}
 2 \langle w'u',v'\rangle  = &- \langle w''u',v\rangle + \langle Pu,(wv)'\rangle  +  \langle w u',Pv\rangle - 2 i \varepsilon \langle wu',v \rangle \\ &+  \int_{\partial \Omega} \left( wu'\partial_\nu \bar v  - \partial_\nu(wu')\bar v - w (Pu)\bar v \nu_x + 2 w'u'\bar v \nu_x\right),
\end{split}\end{equation}
for all $u, \ v \in C_c^\infty(\overline \Omega)$. By density, \eqref{e:ibpuvlong} also holds for all $v \in H^2_{\text{comp}}(\Omega)$. Now let us specialize to the case that $v \in H^2_{\text{comp}}(\Omega) \cap H^1_0(\Omega)$. Then \eqref{e:ibpuvlong} becomes
\begin{equation}\label{e:ibpuvshort}
 2 \langle w'u',v'\rangle  = - \langle w''u',v\rangle + \langle Pu,(wv)'\rangle  +  \langle w u',Pv\rangle - 2 i \varepsilon \langle wu',v \rangle  +  \int_{\partial \Omega} wu'\partial_\nu \bar v.
\end{equation}
Again by density, we also have \eqref{e:ibpuvshort} for all $u \in H^2_{\text{comp}}(\Omega)$ and $v \in H^2_{\text{comp}}(\Omega) \cap H^1_0(\Omega)$.

Specializing further to the case that $u=v$, taking real parts of both sides, and using
\[
- \re \langle w'' u', u \rangle = \frac 12 \langle w''' u, u \rangle, \qquad u'|_{\partial \Omega} = \nu_x \partial_\nu u,
\]
gives \eqref{e:ibpm} for all $u \in \mathcal D_{\text{comp}}$.
To prove \eqref{e:ibpm} for all $u \in \mathcal D$, use a partition of unity to write $u$ as a locally finite sum of functions in $\mathcal D_{\text{comp}}$.
\end{proof}

For the proof of Theorem \ref{t:cig} we will use 
\[w(x) = 1 - (1+x)^{-\delta}, \qquad x \ge 0.\]
We will need the weighted Poincar\'e type inequality
\begin{equation}\label{e:poinc}
  \|\sqrt{w'''}u\| \le   2  \frac{\sqrt{1+\delta}}{\sqrt{2+\delta}} \|\sqrt{w'} u'\|, \text{ for all }u \in H^1_0(\Omega),
\end{equation}
which is proved by writing
 \[
   \|\sqrt{w'''}u\|^2 = - 2 \re \langle w'' u', u \rangle \le 2 \|\sqrt{w'''}u\| \left\|\frac{w''}{\sqrt{w'''}}u'\right\| = 2  \frac{\sqrt{1+\delta}}{\sqrt{2+\delta}} \|\sqrt{w'''}u\|\|\sqrt{w'} u'\|.
\]

\begin{proof}[Proof of Theorem \ref{t:cig}]
We simplify \eqref{e:ibpm},  assuming additionally that
\begin{equation}\label{e:uweigh}
(1+x)^{\frac{3+\delta}2}Pu \in L^2(\Omega).
\end{equation}
Then, since $0<w \le 1$ and the last term of \eqref{e:ibpm} is nonpositive by \eqref{e:xnux}, we have
\begin{equation}\label{e:ibpcigest}
 \|\sqrt{w'} u'\|^2 \le \frac 14 \|\sqrt{w'''}u\|^2 + \left\|\frac{Pu}{\sqrt{w'}}\right\|\|\sqrt{w'}u'\| + \frac 12 \left\|\frac{w'Pu}{\sqrt{w'''}}\right\|\|\sqrt{w'''}u\| + \varepsilon \|u'\|\|u\|.
\end{equation}
We first estimate the last term on the right using \begin{equation}\label{e:ibplap}
 \langle \Delta u, u \rangle = - \|\nabla u\|^2,
\end{equation} which gives
\[
 \|u'\|^2 \le \re \langle Pu,u\rangle + E \|u\|^2 \le \left\|\frac{Pu}{\sqrt{w'''}}\right\|\|\sqrt{w'''}u\| + E \|u\|^2,
\]
and 
\[
 \varepsilon \|u\|^2 = -\im \langle Pu ,u\rangle \le \left\|\frac{Pu}{\sqrt{w'''}}\right\|\|\sqrt{w'''}u\|.
\]
Combining these gives
\[
 \varepsilon^2  \|u\|^2  \|u'\|^2 \le  \left(E + \varepsilon \right)\left\|\frac{Pu}{\sqrt{w'''}}\right\|^2\|\sqrt{w'''}u\|^2,
\]
and plugging into \eqref{e:ibpcigest} gives
\[
  \|\sqrt{w'} u'\|^2 \le \frac 14 \|\sqrt{w'''}u\|^2 + \left\|\frac{Pu}{\sqrt{w'}}\right\|\|\sqrt{w'}u'\| + \frac 12 \left\|\frac{w'Pu}{\sqrt{w'''}}\right\|\|\sqrt{w'''}u\|  + \sqrt{E + \varepsilon }\left\|\frac{Pu}{\sqrt{w'''}}\right\|\|\sqrt{w'''}u\|.
\]
Now we use the weighted Poincar\'e inequality \eqref{e:poinc} to estimate all occurrences of $\|\sqrt{w'''}u\|$ on the right by $\|\sqrt{w'} u'\|$. We then cancel a factor of $\|\sqrt{w'} u'\|$ from all terms, and  move the first term on the right over to the left.
This gives
\[
\frac 1 {2+\delta}  \|\sqrt{w'} u'\| \le \left\|\frac{Pu}{\sqrt{w'}}\right\|+ \frac{\sqrt{1+\delta}}{\sqrt{2+\delta}} \left\|\frac{w'Pu}{\sqrt{w'''}}\right\|  + 2  \frac{\sqrt{1+\delta}}{\sqrt{2+\delta}} \sqrt{E + \varepsilon}\left\|\frac{Pu}{\sqrt{w'''}}\right\|.
\]
Now use $w' \le \delta$ and $w''' \le (1+\delta)(2+\delta)w'$ to combine terms:
\[
 \|\sqrt{w'} u'\| \le 2\sqrt{(1+\delta)(2+\delta)} \left(1+\delta  + \sqrt{E+\varepsilon}\right)\left\|\frac{Pu}{\sqrt{w'''}}\right\|.
\]
Using \eqref{e:poinc} again gives
\[
  \|\sqrt{w'''}u\| \le 4(1+\delta) \left(1+\delta  + \sqrt{E+\varepsilon}\right)\left\|\frac{Pu}{\sqrt{w'''}}\right\|.
\]
Plugging in the formula for $w'''$ and using $\delta \le 1$ to simplify the constants gives
\begin{equation}\label{e:cigapr}
  \|(1+x)^{-\frac{3+\delta}2}u\| \le \frac 3 \delta(1 + \sqrt {E + \varepsilon})\left\|(1+x)^{\frac{3+\delta}2}Pu\right\|,
\end{equation}
for all $u\in \mathcal D$ satisfying \eqref{e:uweigh}. For any $v \in L^2(\Omega)$, by \eqref{e:resmap} we may substitute $u = P^{-1}(1+x)^{-\frac{3+\delta}2}v$ into this last estimate to obtain
\[
  \|(1+x)^{-\frac{3+\delta}2}P^{-1}(1+x)^{-\frac{3+\delta}2}v\|  \le \frac 3 \delta(1 + \sqrt {E + \varepsilon})\|v\|,
  \]
for all $E \ge 0$ and $\varepsilon>0$.

Applying the Phragm\'en--Lindel\"of principle to the functions
\[
 z \mapsto \langle (1+x)^{-\frac{3+\delta}2}(-\Delta - z)^{-1}(1+x)^{-\frac{3+\delta}2}u,v\rangle /(1+\sqrt{-z}), \qquad u, \ v \in L^2(\Omega),
\]
in the sectors
\[
 \{z \in \mathbb C \mid \alpha \re z < | \im z| \}, \qquad \alpha > 0,
\]
(as in e.g. the end of the proof of (1.6) of \cite{cd}) gives
the conclusion.
\end{proof}

\section{Domains  with multiple ends}\label{s:hour}

The proofs of Theorems \ref{t:hour}, \ref{t:flat}, and \ref{t:convexobs} are more elaborate versions of the proof of Theorem \ref{t:cig}.  In comparison with the setting of Theorem \ref{t:cig}, we still have the integration by parts identity \eqref{e:ibpm}, with which we control $u'$. However, we  no longer have $x>0$ throughout $\Omega$, and hence can no longer use the weighted Poincar\'e type inequality \eqref{e:poinc} to estimate $u$  purely in terms of $u'$. To compensate for this, we apply \eqref{e:poinc} to a cut-off version of $u$ in~\eqref{e:ftcpoinc} below. For the proofs of Theorems~\ref{t:hour} and~\ref{t:flat}, we choose the cut-off such that
the resulting remainder term is supported where we have the flaring estimate~\eqref{e:flare}. We then use variants of \eqref{e:ibpm}, proved in Lemmas~\ref{l:ibp2} and~\ref{l:ibp3} below, to control $u$ in the flaring region by $\partial_\nu u$ on the boundary of the flaring region, and then the flaring estimate \eqref{e:flare} and the original integration by parts identity \eqref{e:ibpm} to control $\partial_\nu u$ on the boundary of the flaring region.

We begin by proving the new integrations by parts identities that we will need.

\begin{lem}\label{l:ibp2} Let $\mu = \mu(x)$ be $C^1$, real valued, and bounded with bounded derivative. Let $u \in \mathcal D$, and let $E, \ \varepsilon \in \mathbb R$. Then
\begin{equation}\label{e:ibpe}
 \langle \mu' u', u' \rangle  + E \langle \mu' u, u \rangle =  2 \re \langle \mu Pu, u' \rangle - 2 \varepsilon \im \langle \mu u, u' \rangle + \sum_{j=1}^{d-1}\langle \mu' \partial_{y_j} u, \partial_{y_j} u \rangle + \int_{\partial \Omega} \mu |\partial_\nu u|^2 \nu_x
 \end{equation}

\end{lem}

\begin{proof}
Let $u \in C_c^\infty(\overline{\Omega})$. Adding together the identities
\[\begin{split}
 \langle \mu' u', u' \rangle &= -2 \re \langle \mu u'', u' \rangle + \int_{\partial \Omega} \mu |u'|^2 \nu_x, \\
- \langle \mu' \partial_{y_j} u, \partial_{y_j} u \rangle &=  2 \re \langle \mu \partial_{y_j} u, \partial_{y_j} u' \rangle  - \int_{\partial \Omega} \mu |\partial_{y_j} u |^2\nu_x \\ &= -2 \re \langle \mu \partial_{y_j}^2 u, u' \rangle 
+  \int_{\partial \Omega}\left( 2 \re \mu \partial_{y_j} u \bar u' \nu_{y_j} -\mu |\partial_{y_j} u |^2\nu_x \right),\\
E \langle \mu' u, u \rangle &= - 2 \re \langle \mu E u, u' \rangle + E \int_{\partial \Omega} \mu |u|^2 \nu_x. 
\end{split}\]
gives
\[
\begin{split}
 \langle \mu' u', u' \rangle  + E \langle \mu' u, u \rangle =  &2 \re \langle \mu Pu, u' \rangle - 2 \varepsilon \im \langle \mu u, u' \rangle + \sum_{j=1}^{d-1}\langle \mu' \partial_{y_j} u, \partial_{y_j} u \rangle \\&+ \int_{\partial \Omega}
 \Big( \mu |u'|^2 \nu_x + E\mu |u|^2 \nu_x +\sum_{j=1}^{d-1} \left(2 \re \mu \partial_{y_j} u \bar u' \nu_{y_j} -\mu |\partial_{y_j} u |^2\nu_x \right)\Big).
\end{split}
 \]
By density this holds for $u \in H^2_{\text{comp}}$, and specializing to $u \in \mathcal D_{\text{comp}}$ and using \[
u'|_{\partial \Omega} = \nu_x \partial_\nu u, \qquad \partial_{y_j} u|_{\partial \Omega} = \nu_{y_j} \partial_\nu u,
\]
gives \eqref{e:ibpe} for $u \in \mathcal D_{\text{comp}}$. To prove \eqref{e:ibpe} for all $u \in \mathcal D$, use a partition of unity to write $u$ as a locally finite sum of functions in $\mathcal D_{\text{comp}}$. 
\end{proof}

\begin{lem}\label{l:ibp3}
Let $u , \  y_j u \in \mathcal D$, let $E, \ \varepsilon \in \mathbb R$, and let $j \in \{1, \dots, d-1\}$. Then
\begin{equation}\label{e:ibpy}
 \|\partial_{y_j} u\|^2  = \frac 12 \re \langle Pu, \partial_{y_j} (y_j u) \rangle  +  \frac 12 \re \langle y_j \partial_{y_j} u, P  u \rangle + \varepsilon \im \langle y_j \partial_{y_j} u,   u \rangle + \frac 12 \int_{\partial \Omega}  y_j|\partial_\nu  u|^2 \nu_{y_j} .
\end{equation}
\end{lem}

\begin{proof}
This is proved in the same way as Lemma \ref{l:ibpcig}, but with the commutator $[w \partial_x,\partial_x^2]$ replaced by $[y_j \partial_{y_j}, \partial_{y_j}^2]$.
\end{proof}

\begin{proof}[Proof of Theorem \ref{t:hour}]

Fix $x_0 \in I$ and a closed interval $[x_0-r,x_0+r] \subset I$ and cut-off functions $\chi_0, \ \chi_1, \ \chi_2, \ \chi_3 \in C_c^\infty(I)$, taking values in $[0,1]$, such that $\chi_0 \equiv 1$ near $[x_0-r,x_0+r]$ and  $\chi_{j+1}\equiv 1$ near $\supp \chi_j$ for $j = 0, \ 1, \ 2$. Below we abbreviate $\chi_j(x)$ as $\chi_j$.

We claim that the weighted Poincar\'e inequality \eqref{e:poinc} implies
\begin{equation}\label{e:ftcpoinc}
 \|(1 + |x - x_0|)^{-\frac{3+\delta}2} u\| \le  \frac{2}{2+\delta} \|(1+|x-x_0|)^{-\frac{1+\delta}2}u'\|  + C\|\chi_1 u\|.
\end{equation}
Indeed, by \eqref{e:poinc} we have
\begin{equation}\label{eq:pe2}
 \|(1 + |x-x_0|)^{-\frac{3+\delta}2}(1-\chi_0) u\| \le \frac{2}{2+\delta} \|(1+|x-x_0|)^{-\frac{1+\delta}2}(1-\chi_0)u'\| +   \frac{2}{2+\delta} \|(1+|x-x_0|)^{-\frac{1+\delta}2}\chi_0'u\|,
\end{equation}
which combined with $\|(1 + |x - x_0|)^{-\frac{3+\delta}2}\chi_0 u\| \le \|\chi_1 u\|$ gives \eqref{e:ftcpoinc}. 

Now apply \eqref{e:ibpm}  with $w \in C^3(\mathbb R)$ such that
\begin{itemize}
\item $w'(x)>0$ and $x w(x) \ge 0$ for all $x$,
\item and  $w'(x) = \delta (1+|x-x_0|)^{-1-\delta}$ when $|x - x_0| \ge r$.
\end{itemize}
Note that, with this choice of $w$, and for any fixed $\gamma > 0$, by \eqref{e:poinc} the first term on the right hand side of \eqref{e:ibpm} obeys 
\begin{equation}\label{e:w'''}\begin{split}
\frac 1 4 \langle w''' u, u \rangle &\le \left(\frac {\delta(\delta+1)(\delta+2)}4 + \gamma\right)  \|(1 + |x-x_0|)^{-\frac{3+\delta}2}(1-\chi_0) u\|^2 + C \|\chi_0 u\|^2 \\
 &\le \left(\frac {\delta(\delta+1)}2 + \frac {2\gamma}{\delta+2}\right)  \|(1 + |x-x_0|)^{-\frac{1+\delta}2}(1-\chi_0) u'\|^2 + C \|\chi_1 u\|^2.
\end{split}\end{equation}
As long as $\delta<1$ (which we may assume without loss of generality), we can choose $\gamma$ small enough that 
\[
\frac {\delta(\delta+1)}2 + \frac {2\gamma}{\delta+2} < \delta.
\]
With $\gamma$ so chosen, after plugging \eqref{e:w'''} into \eqref{e:ibpm} we can subtract the first term on the right hand side of \eqref{e:w'''} to the left of \eqref{e:ibpm}, to obtain
\begin{equation}\label{eq:u'andbd}
 \|(1 + |x|)^{-\frac{1+\delta}2} u'\|^2 - \int_{\partial \Omega}
 w |\partial_\nu u|^2 \nu_x \lesssim \langle |Pu|, |u| + |u'| \rangle + \varepsilon \langle |u'|, |u| \rangle + \|\chi_1 u\|^2,
\end{equation}
where, here and below, the implicit constants in $\lesssim$ are uniform for  $u \in \mathcal D$,  $\varepsilon \in (0,1]$, and  $E \gg 1$ large enough. Then, using the fact that $w \nu_x \le 0$ everywhere by \eqref{e:xnux}, and  that \eqref{e:flare} implies $\nu_xw \le -1/C < 0$ on $\supp \chi_3$, we obtain
\begin{equation}\label{e:chi3first}
 \|(1 + |x|)^{-\frac{1+\delta}2} u'\|^2  + \int_{\partial \Omega} |\chi_3 \partial_\nu u|^2 \lesssim \langle |Pu|, |u| + |u'| \rangle + \varepsilon \langle |u'|, |u| \rangle + \|\chi_1 u\|^2.
\end{equation}
Adding \eqref{e:ftcpoinc} gives
\begin{equation}\label{e:chi1right}
 \|(1 + |x|)^{-\frac{3+\delta}2} u\|^2 +  \|(1 + |x|)^{-\frac{1+\delta}2} u'\|^2  + \int_{\partial \Omega} |\chi_3 \partial_\nu u|^2 \lesssim \langle |Pu|, |u| + |u'| \rangle + \varepsilon \langle |u'|, |u| \rangle + \|\chi_1 u\|^2.
\end{equation}

The first two terms on the right side of \eqref{e:chi1right} will be handled later as in the proof of Theorem \ref{t:cig}. To handle the last term apply \eqref{e:ibpe} with $\mu$ chosen nondecreasing such that $x\mu(x) \ge 0$, $\mu'=1$ near $\supp \chi_1$, and  $\supp \mu'\subset \chi_2^{-1}(1)$. After discarding two terms with a favorable sign, that gives
\begin{equation}\label{e:eest}
E \|\chi_1 u\|^2 \le E \langle \mu' u , u \rangle  \lesssim \langle |Pu|, |u'| \rangle + \varepsilon \langle |u|,|u'| \rangle + \sum_{j=1}^{d-1}\|\chi_2 \partial_{y_j} u\|^2,
\end{equation}
which, combined with \eqref{e:chi1right}, gives
\begin{equation}\label{e:partialyright}\begin{split}
  \|(1 + |x|)^{-\frac{3+\delta}2} u\|^2 +  &\|(1 + |x|)^{-\frac{1+\delta}2} u'\|^2  + \int_{\partial \Omega} |\chi_3\partial_\nu u|^2 \lesssim \\ &\langle |Pu|, |u| + |u'| \rangle + \varepsilon \langle |u'|, |u| \rangle + E^{-1}\sum_{j=1}^{d-1}\|\chi_2 \partial_{y_j} u\|^2.
  \end{split}
\end{equation}

Now apply \eqref{e:ibpy}, with $u$ replaced by $\chi_2 u$, and note that $[\partial_{y_j},\chi_2]=0$, to get, for each $j$,
\begin{equation}\label{e:chi20}
 \|\chi_2 \partial_{y_j} u\|^2 \lesssim \langle |\chi_2Pu| +  |[P,\chi_2]u|,|\chi_2u| + |\chi_2\partial_{y_j} u|\rangle  + \varepsilon \langle |\chi_2 \partial_{y_j} u|, |\chi_2 u| \rangle + \int_{\partial \Omega} |\chi_2\partial_\nu u|^2,
\end{equation}
which implies
\begin{equation}\label{e:chi21}
 \|\chi_2 \partial_{y_j} u\|^2 \lesssim \langle |\chi_2Pu|,|\chi_2u| + |\chi_2\partial_{y_j} u|\rangle + \varepsilon \langle |\chi_2 \partial_{y_j} u|, |\chi_2 u| \rangle + \|\chi_3 u\|^2 + \|\chi_3 u'\|^2 + \int_{\partial \Omega} |\chi_2\partial_\nu u|^2.
\end{equation}
This in turn implies, since $\varepsilon \in (0,1]$,
\begin{equation}\label{e:chi22}
 \|\chi_2 \partial_{y_j} u\|^2 \lesssim  \|\chi_2Pu\|^2 + \|\chi_3 u\|^2 + \|\chi_3 u'\|^2 + \int_{\partial \Omega} |\chi_2\partial_\nu u|^2.
\end{equation}
Plugging \eqref{e:chi22} into \eqref{e:partialyright} gives, for $E$ large enough,
\[
   \|(1 + |x|)^{-\frac{3+\delta}2} u\|^2 +  \|(1 + |x|)^{-\frac{1+\delta}2} u'\|^2  + \int_{\partial \Omega} |\chi_3\partial_\nu u|^2 \lesssim \langle |Pu|, |u| + |u'| \rangle + \varepsilon \langle |u'|, |u| \rangle + E^{-1}\|\chi_2 Pu\|^2.
\]
We now estimate the first two terms on the right in the same way that the last three terms in \eqref{e:ibpcigest} were estimated in the proof of Theorem \ref{t:cig}. Then dropping the last two terms on the left gives
\[
 \| (1 + |x|)^{-\frac{3+\delta}2} u\|^2 \lesssim E \| (1 + |x|)^{\frac{3+\delta}2} Pu\|^2.
\]
\end{proof}

\begin{proof}[Proof of Theorem \ref{t:flat}]  Here we use coordinates $(x,y)\in \Real^2$. We begin as in the proof of Theorem \ref{t:hour}, but when we get up to the analogue of \eqref{e:chi3first} we have instead
\[
 \|(1 + |x|)^{-\frac{1+\delta}2} u'\|^2  + \int_{\Gamma_F} |\chi_3 \partial_\nu u|^2 \lesssim \langle |Pu|, |u| + |u'| \rangle + \varepsilon \langle |u'|, |u| \rangle + \|\chi_1 u\|^2,
 \]
 that is, the integral over $\partial \Omega$ is replaced by an integral over $\Gamma_F$. We then proceed as before, up to the analogue of \eqref{e:partialyright}, where we have instead
\begin{equation}\label{e:partialyrightflat}\begin{split}
  \|(1 + |x|)^{-\frac{3+\delta}2} u\|^2 +  &\|(1 + |x|)^{-\frac{1+\delta}2} u'\|^2  + \int_{\Gamma_F} |\chi_3\partial_\nu u|^2 \lesssim \\ &\langle |Pu|, |u| + |u'| \rangle + \varepsilon \langle |u'|, |u| \rangle + E^{-1}\|\chi_2 \partial_{y} u\|^2;
  \end{split}
\end{equation}

At this stage it does not work to apply \eqref{e:ibpy} alone as in the proof of Theorem~\ref{t:hour}, with $u$ replaced by $\chi_2 u$, as this produces a remainder $\int_{\partial \Omega} |\chi_2\partial_\nu u|^2$ on the right which cannot be handled by the $\int_{\Gamma_F} |\chi_3\partial_\nu u|^2$ we have on the left. To deal with this we will remove the part of the remainder over $\partial \Omega \setminus \Gamma_F$ using a multiple of the identity
\begin{equation}\label{e:ak}
0 = \re \langle Pu, \partial_y u\rangle + \varepsilon \im \langle \partial_y u, u \rangle + \frac 12 \int_{\partial \Omega} |\partial_\nu u|^2 \nu_y,
\end{equation}
which is just \eqref{e:ibpy} with the commutator $[y\partial_y,\partial_y^2]$ replaced by $[\partial_y,\partial_y^2] =0$.

More precisely, define cut-offs $\psi_k \in C_c^\infty(\overline{\Omega})$ such that $\psi_k = \chi_2$ on $\Omega_k$ and $\psi_k=0$ otherwise. Multiply $a_k$ by \eqref{e:ak} applied to $\psi_k u$, and subtract the result from \eqref{e:ibpy} applied to $\psi_k u$. Then use $[\psi_k,\partial_y]=0$ as in \eqref{e:chi20} to get
\[
 \|\psi_k \partial_{y} u\|^2 \lesssim \langle |\psi_kPu| +  |[P,\psi_k]u|,|\psi_ku| + |\psi_k\partial_y u|\rangle  + \varepsilon \langle |\psi_k \partial_{y} u|, |\psi_k u| \rangle + \int_{\Gamma_F} |\psi_k\partial_\nu u|^2,
\]
Estimating as in \eqref{e:chi21} and \eqref{e:chi22} gives
\[
  \|\psi_k \partial_{y} u\|^2 \lesssim  \|\psi_k Pu\|^2 + \|\chi_3 u\|^2 + \|\chi_3 u'\|^2 + \int_{\Gamma_F} |\psi_k\partial_\nu u|^2.
\]
Summing in $k$, and plugging into \eqref{e:partialyrightflat}, gives
\[
   \|(1 + |x|)^{-\frac{3+\delta}2} u\|^2 +  \|(1 + |x|)^{-\frac{1+\delta}2} u'\|^2  + \int_{\Gamma_F} |\chi_3\partial_\nu u|^2 \lesssim \langle |Pu|, |u| + |u'| \rangle + \varepsilon \langle |u'|, |u| \rangle + E^{-1}\|\chi_2 Pu\|^2,
\]
for $E$ large enough, after which we conclude as in the end of the proof of Theorem \ref{t:hour}.
\end{proof}

The proof of Theorem \ref{t:convexobs} is a further elaboration of the same ideas. The key point is that in the proofs of Theorems~\ref{t:hour} and \ref{t:flat}  above we did not use \eqref{e:xnux} directly, but rather used it to construct $w$ such that $w \nu_x \ge 0$, $w' >0$, and such that $w' = \delta(1+|x-x_0|)^{-1-\delta}$ away from the flaring set $I$. For a suitable (e.g. symmetric) convex obstacle in a straight planar waveguide this was done in Theorem \ref{t:flat}. For a more general convex obstacle a more complicated construction of $w$ is needed, and the set $I$ will consist of three intervals chosen in the projection of the obstacle onto the $x$-axis and avoiding the points where $\nu_x=0$: see Figure~\ref{f:comarked}.

\begin{figure}[h]
\labellist
\pinlabel $\scriptstyle (x_+,y_+)$ [l] at 450 260 
\pinlabel $\scriptstyle (x_-,y_-)$ [l] at 30 122
\pinlabel $\scriptstyle (x_M,y_M)$ [l] at 382 312 
\pinlabel $\scriptstyle (x_m,y_m)$ [l] at 110 65
\pinlabel $\scriptstyle x_1$ [l] at 113 5
\pinlabel $\scriptstyle x_2$ [l] at 155 5
\pinlabel $\scriptstyle x_3$ [l] at 260 5
\pinlabel $\scriptstyle x_4$ [l] at 370 5
\pinlabel $\scriptstyle x_5$ [l] at 413 5
\endlabellist
\includegraphics[width=10cm]{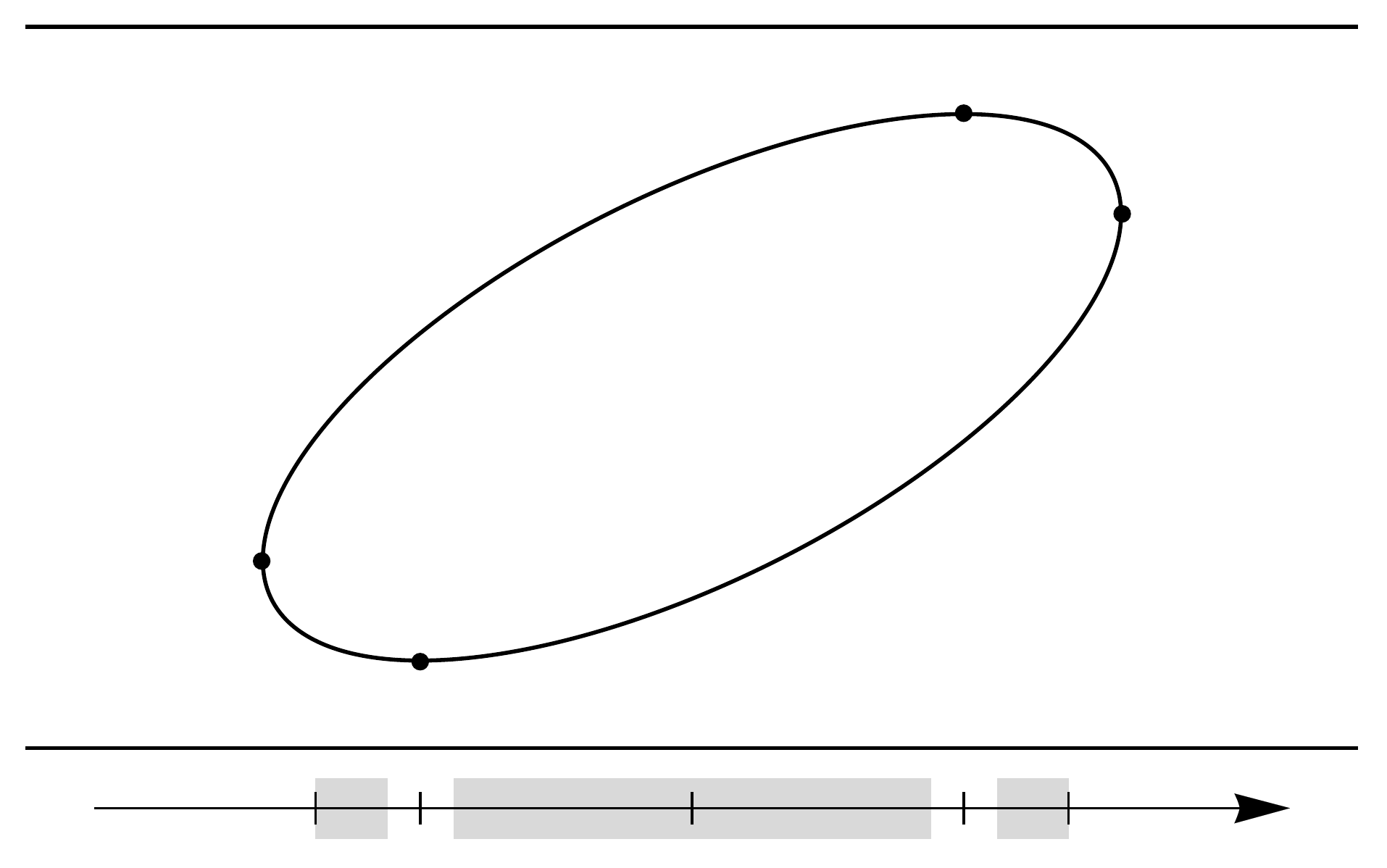}
 \caption{Notation used in the proof of Theorem \ref{t:convexobs}. The shaded gray subset of $\mathbb R$ is $K$, and $I$ is a small neighborhood of $K$.}\label{f:comarked}
\end{figure}
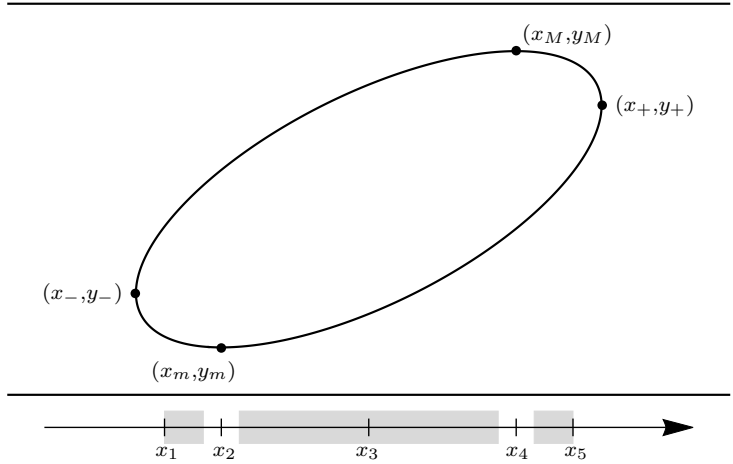

The function $w$ will be constructed from preliminary functions $w_+$ and $w_-$ adapted to the upper and lower parts of $\Omega$ respectively. 
\begin{lem}\label{l:wpm}
Let $\delta>0$, and let $-\infty<x_1<x_2<x_3<x_4<x_5<\infty$.  Then
there are functions $w_\pm \in C^3(\Real)$ so that:
\begin{itemize}
\item  For all $x\in \Real$, $w_\pm'(x)>0$.
\item If $x<x_1$ or $x>x_5$, then $w_+(x)=w_-(x)$ and 
$w_+'(x)=w_-'(x)= \delta(1+|x-x_3|)^{-\delta -1}$.
\item Each of $x_2$ and $x_4$ is contained in an open interval on which \\ 
$w_+'(x)=w_-'(x)= \delta(1+|x-x_3|)^{-\delta -1}.$
\item The equalities $w_+(x_4)=0$ and $w_-(x_2)=0$ hold.
\end{itemize}
\end{lem}
\begin{proof}
We will define $w_\pm$ via $w_+(x)=\int_{x_4}^xw_+'(t)dt$ and 
$w_-(x)=\int_{x_2}^xw_-'(t)dt$ in order to satisfy the last condition.
Let $\rho_0=\frac{1}{3}\min_{j=1,2,3,4}(x_{j+1}-x_j)$.  Set
\begin{equation}
\label{eq:wppmdef}
w_{\pm}'(x)=\left\{ \begin{array} {ll}
\delta(1+|x-x_3|)^{-\delta -1} & \text{if}\; x<x_1\; \text{or } x>x_5\\
\delta(1+|x-x_3|)^{-\delta -1}& \text{if}\; |x-x_2|<\rho_0\; \text{or } \; |x-x_4|<\rho_0\\
h_{\pm,1}(x)& \text{if } \; x_1\leq x\leq x_2-\rho_0\\
h_{\pm,2}(x)& \text{if } \; x_2+\rho_0 \leq x \leq x_4-\rho_0\\
h_{\pm,3}(x)& \text{if } \; x_4+\rho_0 \leq x\leq x_5
\end{array}
\right.
\end{equation}
for some $h$'s yet to be chosen.
Now choose strictly positive
 $h_{\pm,1}\;h_{\pm,2},\; h_{\pm,_3}$ so that the resulting functions
$w_{\pm}'$ as defined above are $C^2$, $w_\pm'(x)>0$ for all $x$, and 
so that $\int_{x_2}^{x_1}w_-'(t)dt=\int_{x_4}^{x_1}w_+'(t)dt$
and $\int_{x_2}^{x_5}w_-'(t)dt=\int_{x_4}^{x_5}w_+'(t)dt$.  The 
conditions on the integrals guarantee that $w_+(x)=w_-(x)$ if $x<x_1$ or
$x>x_5$.  Satisfying this condition on the integrals 
 may be accomplished by first choosing $h_{\pm,2}$, $h_{-,3}$,
and $h_{+,1}$, and then choosing $h_{+,3}$ and $h_{-,1}$ so that the integral
conditions are satisfied.
\end{proof}

\vspace{2mm}
\noindent
{\em Proof of Theorem \ref{t:convexobs}}.  
We begin by naming the coordinates of certain points on $\partial\mco$ as in
Figure~\ref{f:comarked}.
Set $y_M=\max\{y:(x,y)\in \overline{\mco}\; \text{for some $x\in \Real$}\}$,
$y_m=\min\{y:(x,y)\in \overline{\mco}\; \text{for some $x\in \Real$}\}$, and 
let $(x_M,y_M)$, $(x_m,y_m)$ denote the corresponding points on $\partial \mco$.
Likewise, set $x_\pm = \pm \max\{ \pm x: (x,y)\in \overline{\mco}\; \text{for some $y\in \Real$}\}$, and let $(x_\pm,y_\pm)$ be the corresponding 
points in $\partial \mco$.  By the strict convexity of $\mco$, each of 
these points is uniquely defined.  Without loss of generality, we may assume
$x_M\geq x_m$.  Since the case $x_M=x_m$ is covered by Theorem \ref{t:flat}, we
assume for the remainder of the proof that $x_M>x_m$. 

Let $r_1=\frac{1}{3}\min(x_+-x_M, x_M-x_m, x_m-x_-)$, and set 
\begin{equation}
x_1=x_-+r_1,\; x_2=x_m, \:x_3= \frac{x_m+x_M}{2}, \; x_4= x_M, \; x_5= x_+-r_1.
\end{equation} 
With these choices, set $w_\pm$ to be the functions given in Lemma 
\ref{l:wpm}.  We shall use these $w_\pm$ to define a single function 
$w$ on $\Omega$, which is adapted to account for the fact that the ``highest'' and 
``lowest'' points of $\partial \mco$, $(x_M, y_M)$ and $(x_m,y_m)$, have
different $x$ coordinates.  Otherwise, our function $w$ will be very similar
to the weights $w$ we have used earlier.

We can write 
$$\Omega\setminus \large( \large((-\infty, x_-)\times \{y_-\}\large) \cup
\large( (x_+,\infty)\times\{y_+\}\large)\large)=\Omega_+\cup \Omega_-$$
where $\Omega_\pm$ are disjoint  connected open sets.  We label these so
that $(-\infty,x_-)\times (y_-,1)\subset \Omega_+$; that is, $\Omega_+$ is the 
``upper'' of the two components.  Now define
\begin{equation}
w(x,y)=\left\{ \begin{array}{ll}
w_+(x) & \text{if $(x,y)\in \Omega_+$}\\
w_-(x)& \text{if $(x,y)\in \Omega_-$}\\
w_+(x)& \text{if $(x,y)\in (-\infty,x_-)\times\{y_-\}\;\text{or if}\; (x,y)\in (x_+,\infty)\times\{y_+\}$}.\end{array}
\right.
\end{equation}
By our choice of $w_\pm$, $w$ is $C^3$ on $\Omega$.  Moreover, $w\nu_x\leq 0$ on 
$\partial \mco$, with equality only at the points $(x_m,y_m)$ and $(x_M,y_M)$.

We now claim that, for $u\in \mathcal D$,
(\ref{e:ibpm}) holds, even though our $w$ is not 
independent of $y$.  To see this, with 
$u,v\in C_c^\infty(\overline{\Omega})$ apply (\ref{e:ibpuvlong}) on $\Omega_\pm$ with 
$w$ replaced by $w_\pm$, and then add the resulting equalities. The boundary 
terms involving 
$\partial \Omega_\pm \setminus \left( \partial \Omega \cap \partial \Omega_\pm\right)$ sum to $0$.  The remainder of the proof follows as in the proof of Lemma \ref{l:ibpcig}. Alternatively, observe that the proof of Lemma \ref{l:ibpcig} only used the fact that $\partial_y w \equiv 0$, and not that $w$ is independent of $y$.

Set 
$$K=\bigcup_{\pm}\supp (w_\pm'(x)-\delta(1+|x-x_3|)^{-1-\delta}).$$
and note that by our choice of $w_\pm$, there is a $C_K>0$ so that  
$$w\nu_x\leq -C_K<0\; \text{on $(K\times(-1,1))\cap \partial \mco$}.$$
Moreover, there is an open set $I$, 
$\overline{I}\subset (x_-,x_+)\subset \Real$, $K\subset I$, and 
a $C_I>0$ so that 
$$w\nu_x\leq -C_I<0\; \text{on $(I\times(-1,1))\cap \partial \mco$}.$$
The use of this $I$ is very similar to the use of $I$ of Theorem \ref{t:flat}.
Moreover, the union of curves $(I\times(-1,1))\cap \partial \mco$ plays a role similar to that
of $\Gamma_F$ from Theorem \ref{t:flat}.

Now choose $\chi_0,\;\chi_1,\;\chi_2,\;\chi_3\in C_c^\infty(I)$, taking
values in $[0,1]$, so that $\chi_0\equiv 1$  near $K$ and 
$\chi_{j+1}\chi_j=\chi_j$ for $j=0,\;1,\;2$. 

 Next, we note that (\ref{eq:pe2})
is valid for our $\chi_0$, which in turn implies
that (\ref{e:ftcpoinc}) is valid for our $\chi_1$.
Then, just as in the proof of 
(\ref{eq:u'andbd}), if $\delta<1$ we can show that
$$\|(1 + |x|)^{-\frac{1+\delta}2} u'\|^2 - \int_{\partial \Omega} w |\partial_\nu u|^2 \nu_x \lesssim \langle |Pu|, |u| + |u'| \rangle + \varepsilon \langle |u'|, |u| \rangle + \|\chi_1 u\|^2$$
and 
$$-\int_{\partial \Omega}w |\partial_\nu u|^2\nu_x\geq
 C_I \int_{(I\times(-1,1))\cap\partial \mco} |\partial_\nu u|^2.$$
Proceeding as in the proofs of Theorems \ref{t:hour} and \ref{t:flat}, we get to  the analogue of \eqref{e:chi1right}:
\[\begin{split}
 \|(1 + |x|)^{-\frac{3+\delta}2} u\|^2 +  \|(1 + |x|)^{-\frac{1+\delta}2} u'\|^2 & + \int_{(I\times(-1,1))\cap\partial \mco} |\chi_3 \partial_\nu u|^2 \\&\lesssim \langle |Pu|, |u| + |u'| \rangle + \varepsilon \langle |u'|, |u| \rangle + \|\chi_1 u\|^2.
\end{split}\]
We now apply \eqref{e:ibpe} with $\mu\in C^1(\Omega)$ chosen such that $\partial_y \mu \equiv 0$, $\mu(x_m,y_m) = \mu(x_M,y_M) = 0$, $\mu'\geq 0$, $\mu' = 1$ near $\supp \chi_1$, and $\supp \mu' \subset \chi_2^{-1}(1)$. The construction of such a $\mu$ follows along the same lines as (but is simpler than) the construction of $w$ above. That gives \eqref{e:eest}, after which the proof proceeds just like the proof of Theorem~\ref{t:flat}.
\qed

\section{Absence of eigenvalues and embedded resonances}\label{s:noembed}
For the proof of Theorem \ref{t:noembed}, we use a variant of Lemma \ref{l:ibpcig}. For $R>0$, let 
\[\Omega_R\defeq\{ (x,y)\in \Omega: \; |x|<R\}.\]
Denote
$\|v\|^2_{L^2(\XR)}=\int_{\XR}|v|^2$ and 
$\langle u, v\rangle_{\XR}=\int _{\XR}  u \overline{v}$.  
\begin{lem}\label{l:ibpforeigenvalues}
Let $u\in \mathcal D$, $R>0$, and let $E, \ \varepsilon \in \mathbb R$.  Then
with $\nabla_y$ denoting the gradient in the $y$ variables only,
\begin{multline}\label{eq:xrint1}
\|u'\|^2_{L^2(\XR)}=
\frac{1}{2}\re \langle Pu, (xu)'\rangle_{\XR}+ \frac{1}{2} \re \langle xu', Pu
\rangle_{\XR} +\varepsilon\im \langle x u',u\rangle_{\XR}+
\frac{1}{2}\int_{\partial \Omega\cap \partial \XR}  x |\partial_\nu u|^2 \nu_x
\\+\frac 12\sum_{\pm}  \re \int_{\partial \XR\cap\{x=\pm R\}}
\left( \pm u'\overline{u}+R(-|\nabla_y u|^2+E|u|^2+|u'|^2)\right).
\end{multline}
Moreover,
\begin{equation}\label{eq:xrint2}
0 = \im \langle Pu, (xu)'\rangle_{\XR}+\im \langle xu',Pu\rangle_{\XR}
-2\varepsilon \re \langle xu',u\rangle_{\XR}+
\sum_{\pm} \int_{\partial \XR\cap\{x=\pm R\}} \varepsilon R |u|^2 \pm  \im u'\overline{u}.
\end{equation}
\end{lem}
\begin{proof}
The proof of (\ref{eq:xrint1}) 
is essentially that of Lemma \ref{l:ibpcig}.  In particular, 
we use (\ref{e:ibpuvlong}), replacing $\Omega$ by $\XR$.
In addition, we use $w(x)=x$.  Then following
the outline of  Lemma \ref{l:ibpcig} and taking real parts
gives (\ref{eq:xrint1}).   The equality (\ref{eq:xrint2}) follows 
from the same argument, but taking the imaginary part of the resulting 
equation, rather than the real part.
\end{proof}

\begin{proof}[Proof of Theorem \ref{t:noembed}]
We give a proof by contradiction.  Suppose the Dirichlet Laplacian
on $\Omega$ has an eigenvalue $E_1$ or a
resonance embedded in the continuous spectrum at $E_1$.
 Let $u$ be an associated eigenfunction or outgoing resonance state.
Then by separation of variables there are  constants $\gamma_j$
so that 
\begin{equation}\label{eq:threshres}
u(x,y)=\sum_{\sigma_j^2< E_1}\gamma_j e^{i|x|\sqrt{E_1-\sigma_{j}^2}}\phi_j(y)+
\sum_{\sigma_j^2\geq E_1}\gamma_j e^{-|x|\sqrt{\sigma_{j}^2-E_1}}\phi_j(y), \qquad |x| \ge R_0,
\end{equation}
with notation as in  (\ref{eq:cylend}) and \eqref{e:phisig}. To see \eqref{eq:threshres}, recall that the outgoing condition says that   
\begin{equation}\label{e:outgoing}
u|_{\{\pm x > R_0\}} = \left(\lim_{\varepsilon \downarrow 0}(-\Delta_\pm - E_1 - i \varepsilon)^{-1} f_\pm\right) \Big|_{\{\pm x > R_0\}},
\end{equation}
for some $f_\pm \in L^2_{\text{comp}} (\mathbb R_\pm \times Y_\pm)$, where $-\Delta_\pm$ is the Dirichlet resolvent on $\mathbb R_\pm \times Y_\pm$. (The choice of outgoing condition is not essential, and one could also use an incoming condition, replacing $i$ by $-i$ in \eqref{e:outgoing} and in \eqref{eq:threshres}). Furthermore, by \eqref{e:outgoing}, the sum over $\sigma_j^2 \ge E_1$ in \eqref{eq:threshres} converges in $L^2(\Omega \cap \{|x|>R_0\})$, and hence we have
\begin{equation}\label{e:gammabound}
\sum_{\sigma_j^2> E_1} 
|\gamma_j|^2 e^{-2R_0\sqrt{\sigma_j^2-E_1}} (\sigma_j^2-E_1)^{-1/2}< + \infty.
\end{equation}

We first prove that $E_1$ cannot be an 
eigenvalue, since this is easier.

Suppose $E_1$ is an eigenvalue, so that  $u$ is an associated
eigenfunction. Since $u \in L^2$, we must have $\gamma_j = 0$ when $\sigma_j \le E_1$. By \eqref{e:gammabound}, $u$ and its derivatives tend to $0$  exponentially in $|x|$.

Using Lemma \ref{l:ibpforeigenvalues} with $E=E_1$
and $\varepsilon=0$ we find
\begin{equation}\label{e:u'R}\begin{split}
\|u'\|^2_{L^2(\XR)}=
&\frac{1}{2}\int_{\partial \Omega\cap \partial \XR}  x |\partial_\nu u|^2 \nu_x
 \\&+ \frac 12 \sum_{\pm}  \re \int_{\partial \XR\cap\{x=\pm R\}}
\left( \pm u'\overline{u}+R(-|\nabla_y u|^2+E_1|u|^2+|u'|^2)\right).
\end{split}\end{equation}
Taking the limit in   as $R\rightarrow \infty$ and using the 
exponential decay of $u$ and its derivatives gives
$$\|u'\|^2_{L^2(\Omega)}=\frac{1}{2}\int_{\partial \Omega}  x |\partial_\nu u|^2 \nu_x \leq 0.$$
But this means that $u$ is independent of $x$.  Since $u\in L^2(\Omega)$ 
is nontrivial, this
is a contradiction.

The argument for showing there are no resonances embedded in the continuous
spectrum is similar, but requires some further  computations.

Suppose $u$ is a resonance state associated to $E_1\in \Real$.  
Applying (\ref{eq:xrint2}) with $E=E_1$, $\varepsilon=0$
and using (\ref{eq:threshres}) along with the orthonormality of $\{\phi_j\}$
gives, for sufficiently large $R$,
$$0= \im \sum_{\pm}\pm \int_{\partial\XR\cap\{x=\pm R\}}  u'\overline{u} =
\sum_{\sigma_j^2<E_1} \sqrt{E_1-\sigma_{j}^2}|\gamma_j|^2$$
which in turn implies that $\gamma_j=0$ if $\sigma_j^2<E_1$.

Returning to (\ref{eq:threshres}), 
note again that by \eqref{e:gammabound} the terms with $\sigma_j^2>E$ are
exponentially decaying in $|x|$ along with their derivatives, while those with 
$\sigma_j^2=E_1$ have $x$ derivative $0$.
From (\ref{eq:threshres}) and using these observations, 
\begin{multline*}
\int_{\partial \XR\cap\{x=\pm R\}}
\left( \pm u'\overline{u}+R((-|\nabla_y u|^2+E_1|u|^2+|u'|^2)\right)\\
=\int_{\partial \XR\cap\{x=\pm R\}}
\left( \pm u'\overline{u}+R((\Delta_y u) \overline{u}+E_1|u|^2+|u'|^2)\right)
\end{multline*}
is exponentially decreasing in $R$.
Again taking  the limit of \eqref{e:u'R} as $R\rightarrow \infty$ 
we have
$$\|u'\|^2_{L^2(\Omega)}=\frac{1}{2}\int_{\partial \Omega} x |\partial_\nu u|^2 \nu_x 
\leq 0,$$
so that $u'\equiv 0$.  But a nontrivial $u$ with $u'\equiv 0$ and $-\Delta u=E_1 u$
 can only satisfy Dirichlet boundary conditions on $\partial \Omega$ if  $\partial \Omega$ is invariant under translation in the $x$ direction; 
that is,
$\Omega=\Real \times \widetilde{Y}$ for some $\widetilde{Y}\subset \Real^{d-1}$.
\end{proof}

\section{Resonance-free regions}\label{s:resfree}
For a domain $\Omega\subset \Real^d$ which has  cylindrical ends, 
the resolvent of the Dirichlet Laplacian $(-\Delta - z)^{-1}$ has a meromorphic continuation
to a Riemann surface $\hat{Z}$.  
Resolvent estimates of the type of Theorems \ref{t:cig}, \ref{t:hour}, \ref{t:flat}, and \ref{t:convexobs} imply, essentially via an
application of \cite[Theorem 5.6]{cd}, that there is a region near the continuous spectrum 
in which the meromorphic continuation of the resolvent is in fact analytic.
To make a precise statement, we first introduce the space to which
the resolvent continues.

The continuous spectrum of $-\Delta$ is given by $[\sigma_1^2,\infty)$, where $\sigma_1^2$ is the smallest Dirichlet eigenvalue of $-\Delta_Y$.
For general domains or  manifolds with 
cylindrical ends, there may, in addition, be  eigenvalues of $-\Delta$, either in $(0,\sigma_1^2)$ or 
embedded in the continuous spectrum.
  For $z\in \Complex $ so that $z$ is not in the spectrum of $-\Delta$ set
$R(z)=(-\Delta-z)^{-1}:L^2(\Omega)\rightarrow L^2(\Omega)$. As an operator
from $L^2_{\text{comp}}(\Omega)$ into $L^2_{\operatorname{loc}}(\Omega)$, the resolvent $R(z)$
has a meromorphic continuation to the Riemann surface
$\hat{Z}$ which we describe next.

The Riemann surface $\hat{Z}$  is determined by the set $\{\sigma_j^2\}$
of Dirichlet eigenvalues of $-\Delta_Y$.
For $z\in \Complex \setminus [\sigma_1^2,\infty)$,
define $\tau_j(z)=(z-\sigma_j^2)^{1/2}$, where we take the square root to 
have positive imaginary part.  Then $\hat{Z}$ is the minimal Riemann 
surface so that for each $j\in \Natural$, $\tau_j(z)$ is an analytic,
single-valued function on $\hat{Z}$. The Riemann surface $\hat{Z}$ forms
a countable cover of $\Complex$, ramified at points corresponding to 
$\sigma_j^2$, $j\in \Natural$.  For any $z\in \hat{Z}$, $\imag \tau_j(z)>0$ for all but
finitely many $j$.  We call 
the ``physical region'' the portion of $\hat{Z}$
in which $\imag \tau_j(z)>0$ for all $j\in \Natural$.  In the 
physical region and away from eigenvalues of $-\Delta$, 
$R(z)$  is a bounded operator on $L^2(\Omega)$.
 For further details about the construction of
$\hat{Z}$ and a proof that the
resolvent of $-\Delta$ on $\Omega$ has a meromorphic continuation to $\hat{Z}$, 
see \cite{gui}, \cite[Section 6.7]{tapsit}, or \cite[Section 2]{cd2}.

We define a distance on $\hat{Z}$ as follows: for $z,\;z'\in \hat{Z}$,
\begin{equation}
d(z,z')\defeq \sup_{j}| \tau_j(z)-\tau_j(z')|.
\end{equation}
That this is  a metric is shown in \cite[Section 5.1]{cd}.

For $E>|\sigma_1|$, denote by $E\pm i0$ the points in $\hat{Z}$ which are on 
the boundary of the physical region and which are obtained as limits $\lim_{\pm \delta \downarrow 0}E+i\delta$. These
points correspond to the continuous spectrum of $-\Delta$.
If $E>\sigma_j^2$, then $\pm \tau_j(E\pm i 0)>0$,
 and if $\sigma_j^2>E$ then $\tau_j(E\pm i 0)\in i \Real_+$.

The next theorem describes quantitatively a region near the boundary
of the physical space in which the resolvent is guaranteed to be analytic.
\begin{thm}\label{t:resfree}
Let $\Omega\subset \Real^d$ be a domain with  cylindrical 
ends which in addition satisfies the conditions of one of 
Theorem \ref{t:cig}, \ref{t:hour}, \ref{t:flat}, or \ref{t:convexobs}, and 
let $\chi \in L^\infty_{\text{comp}}(\Omega)$.  Then there
are positive constants $C_1, \ C_2$, and $E_0$ so that
$\chi R(z)\chi$ is analytic in
 $\{ z\in \hat{Z}: \; d(z,E\pm i0)<C_1(1+E)^{-1}\}$ for all $E \ge E_0$, and in this same 
region $\| \chi R(z)\chi\| \leq C_2 (1+E)^{1/2}$.
\end{thm}

After a semiclassical rescaling, the proof of this theorem is the same as the
proof of \cite[Theorem 5.6]{cd}. More specifically, we write $(-\Delta - E) = h^{-2}(-h^2 \Delta - 1)$ with $h = E^{-1/2}$. Then the $O(E^{1/2})$ resolvent bound implied by Theorem \ref{t:cig}, \ref{t:hour}, \ref{t:flat}, or \ref{t:convexobs} corresponds to a $O(h^{-3})$ resolvent bound for the scaled operator.


\vspace{2mm}
\noindent
{\bf Acknowledgments.}   The authors gratefully acknowledge the 
partial support of the Simons Foundation (TC, collaboration grant for
mathematicians), an MU Research Leave (TC),
 and the National Science Foundation (KD,  Grant DMS-1708511).
The authors thank Peter Hislop for helpful conversations.

\end{document}